\documentclass[UTF8]{article}[12pt]
\usepackage{mathrsfs}
\usepackage{mathtools}
\usepackage{graphicx}
\usepackage{epstopdf}
\usepackage{stmaryrd}
 \usepackage{subfigure}
\usepackage{float}
\usepackage{amsmath}
    \usepackage{bm}
\usepackage{amsfonts,amssymb}
  \usepackage{dsfont}
\usepackage{amsfonts}
\usepackage{mathrsfs}
  \usepackage{pifont}
\numberwithin{equation}{section}
 \usepackage{amsthm}
 \newtheorem{theorem}{Theorem}[section]
\newtheorem{lemma}[theorem]{Lemma}

 \usepackage{amsmath}
 \usepackage{hyperref}
 \usepackage{multirow}

\usepackage{algorithmic}
\usepackage{algorithm}

\begin{document}
\title{\bf A Neumann interface optimal control problem with elliptic PDE constraints and its discretization and numerical analysis }

\author{Zhiyue Zhang$^1$, Kazufumi Ito$^2$ and Zhilin Li$^2$ \thanks{1. School of Mathematical Sciences, Jiangsu Key Laboratory for NSLSCS,
Nanjing Normal University, Nanjing 210023, China. (zhangzhiyue@njnu.edu.cn). 2. Department of Mathematics,
North Carolina State University, 27695, USA}}

\date{}
\maketitle
\textbf{Abstract.}
We study an optimal control problem governed by elliptic PDEs with interface, which the control acts on the interface. Due to the jump of the coefficient across the interface and the control acting on the interface, the regularity of solution of the control problem is limited on the whole domain, but smoother on subdomains. The control function with pointwise inequality constraints is served as the flux jump condition which we called Neumann interface control. We use a simple uniform mesh that is independent of the interface. The standard linear finite element method can not achieve optimal convergence when the uniform mesh is used. Therefore the state and adjoint state equations  are discretized by piecewise linear immersed finite element method (IFEM). While the accuracy of the piecewise constant approximation of the optimal control on the interface is improved by a postprocessing step which possesses superconvergence properties; as well as the variational discretization concept for the optimal control is used to improve the error estimates. Optimal error estimates for the control, suboptimal error estimates for state and adjoint state are derived. Numerical examples with and without constraints are provided to illustrate the effectiveness of the proposed scheme and correctness of the theoretical analysis.


\textbf{AMS:} 49J20, 49M05, 65M60, 65N30, 35Q93.
~\\

\section{Introduction}{\label{section1}}

Optimal control problems governed by elliptic PDEs with interfaces arise in many applications, such as the optimization or optimal control of a process in a domain composed of several materials separated by curves or surfaces (called interfaces), or of a porous medium and an adjacent free-flow region appearing in a wide range of industrial, medical, atmospheric and environmental applications\cite{pinnau2008optimization,lions1971optimal,mu2006conditional,ito2008lagrange,HU2000Boundary}. Coefficients in the elliptic PDEs may have a jump across the interface corresponding to different materials and media\cite{baber2012numerical,babuvska1970finite,vassilev2009coupling}. People pay more attention to the interface conditions that are to be determined such that the response of a physical or engineering system is optimal in some sense. This problem can often be summarized as the minimization of a cost functional subject to partial differential equations with interface. We call this kind of control problems the interface control problem. It has more important applications in science and engineering areas. For example, evaporation is an important process in the climatic or synoptic system because evaporation rates and patterns affect the energy balance of terrestrial or oceanic surfaces and drive a multitude of climatic and synoptic process. Predictability of evaporative rates remains a challenge due to the radiation, humidity, temperature, air velocity, turbulent conditions at the interface. Other examples involve groundwater, flow in fuel cells, turbulence, crystal growth, blood vessels, semiconductor materials and complex material problems\cite{rosenzweig2007laminar,muthing2012dune,dubljevic2010boundary,kauffmann2000optimal,layton2002coupling,
cao2010finite,gunzburger1991analysis,kafafy2005three}. Hence, it is a challenge to develop efficient numerical methods for such interface optimal control problems. In \cite{zhang2015immersed} we show the case of distributed control with interface PDE constraints, which is the easier one with respect to the mathematical analysis. The analysis for interface control problem is more difficult since the regularity of the state function is lower than that for distributed controls. Moreover the internal approximation of the domain causes problems. There are some research results for error estimates of the Neumann boundary control problem and the regularity of the solutions for the Dirichlet boundary control problem in polygonal domains in \cite{Deckelnick2009Finite, apel2015finite, apel2015regularity} and the references therein.  To simplify the analysis, we assume here that the domain is a convex polygonal domain. Although this makes things or problems simpler, the low regularity of state on a convex polygonal domain and the discontinuity across the interface and coefficients are still very complicated to study or solve numerically. Here all main variants of elliptic problems have been studied such as distributed control\cite{gunzburger1991analysis,li2002adaptive,he2023novel,lions1971optimal,meyer2006optimal,vexler2008adaptive,hintermuller2009pde,allendes2022error}, boundary control\cite{hinze2009note,krumbiegel2010priori,casas2005error,apel2015finite,chen2021norm,vexler2007finite,
apel2012finite,casas2012approximation,mateos2011saturation,Chowdhury2017,Wachsmuth2016,Xie}, distributed observation and boundary observation\cite{beuchler2013boundary,sano2011neumann,krumbiegel2010priori}. Applying postprocessing \cite{meyer2004superconvergence} and variational discretization techniques\cite{hinze2005variational}, we obtained an error estimate between a locally optimal and the numerical solution in the $L^2$ norm. This estimate holds for a piecewise linear immersed finite element method and piecewise constant control functions.

Elliptic interface problems have been extensively discussed in the literatures\cite{pan2021high,Hansbo2002An,adjerid2023enriched,chen1998finite,li2017accurate,guo2021error,Wang2021New}. There are different methods such as finite difference method\cite{pan2021high,chen2023arbitrarily}, finite element method\cite{gong2008immersed,chen2023semi,li2003new,lin2007error,hou2010numerical}, finite volume element method\cite{ewing1999immersed,zhu2015immersed,he2009bilinear,wang2021bilinear}, and penalized finite element method\cite{lin2015partially} to solve  these problems. We use a uniform Cartesian mesh in our method. How to design  accurate methods on unfitted meshes has attracted a lot of attention in the literature. The immersed finite element method (IFEM) proposed in \cite{li2003new} is among a few methods that based on linear finite element discretizations and unfitted meshes, for example, uniform triangulations. The idea of the IFEM is to modify the basis functions in the interface triangles so that the interface conditions are satisfied. Optimal approximation capabilities of the immersed finite element space have been proved in \cite{li2004immersed}. And optimal error estimates in $L^2$ and $H^1$ norms have been given in \cite{gong2008immersed}.

Finite element approximations of optimal control problems with PDE constraints are important for the numerical treatment of optimal control problems related to practical applications. Numerical methods for optimal control problems governed by elliptic PDEs have been discussed in many publications (see, e.g., \cite{pinnau2008optimization,troltzsch2010optimal,liu2008adaptive,gong2011mixed,krumbiegel2010priori,vexler2007finite,
apel2012finite,casas2021numerical,sano2011neumann,casas2012error,maurer2000optimization,
lee2000analysis,meyer2008error} and therein references). An overview on the numerical a priori and a posteriori analysis for elliptic control problems can be found in \cite{pinnau2008optimization,troltzsch2010optimal,liu2008adaptive}. However, to the best of our knowledge, there are few papers that concern the numerical method based on unfitted meshes for the Neumann interface optimal control problems governed by elliptic PDEs with discontinuous coefficient.

In this paper, we study the Neumann interface control problem with PDE constraints. Because the control acts on the interface, and the flux jump condition is nonhomogeneous, the optimal control solution has lower regularity on whole domain. Thus this results in difficulties on both theoretical analysis and numerical computations for the interface control problem. Because the presence of control constraints on the interface makes the arising first order necessary conditions non-smooth, it is so hard to use higher order discretization techniques, thus the interface control problem governed by elliptic interface problems is discretized by means of the piecewise linear IFEM on uniform triangulations for the state and adjoint state, while for control, only piecewise constant. We prove suboptimal convergence rates for the state and adjoint state variables. When one uses piecewise constant approximations of the control, the expected order of approximation of the state is greater than $1$. However we prove the convergence order of the state for interface control problem with elliptic PDE constraints with discontinuous coefficient is $3/2$. This key results make us to the same kind of superconvergence for the approximations of corresponding adjoint state for the interface control problem. It is possible to obtain second order convergence of error estimates for elliptic distributed optimal control problems by using the variational discretization concept\cite{hinze2005variational} and postprocessing approach\cite{meyer2004superconvergence}. A better approximation for the error estimates of the control on the interface has been proved based on considering both of approaches of improvement approximations order. In the case of the control without constraints, the discretization leads to a symmetric but indefinite system of equations. A block diagonally preconditioned MINRES algorithm \cite{paige1975solution,rees2010optimal} can be used to solve the indefinite system. In the case of the control with constraints, a nonlinear and non-smooth equation for the discrete control is obtained. A fix-point iteration or the semi-smooth Newton algorithm can be used to solve that single equation for the discrete control.  Optimal error estimates for the control, suboptimal error estimates for the state and adjoint state, are derived, which are the same as that of a Neumann boundary control problem without interfaces.

The paper is organized as follows: In Section~\ref{sec_model}, the model problem is introduced and optimality conditions and regularity results of the problem are given. Section~\ref{num} presents the discretization of the interface optimal control problem based on piecewise constant for the control on the interface and the linear immersed finite element method for both of the state and adjoint state. And some error estimates for the control on the interface, the state and the adjoint state for the interface control problem are derived in Section~\ref{error}. Some conclusions are made in Section \ref{sec_con}.

\section{Model problem}{\label{sec_model}}
\subsection{Constraint equation}{\label{consequ}}
Consider the elliptic interface problem,
\begin{eqnarray}\label{state1}
&&-\nabla \cdot (\beta(\mathbf{x})\nabla y(\mathbf{x}))=f(\mathbf{x}), \quad\mbox{ in }\Omega\backslash\Gamma,\\ \label{state2}
&&\left[y(\mathbf{x})\right]_{\Gamma}=0,~~\left[\beta\partial_\mathbf{n} y(\mathbf{x})\right]_{\Gamma}=u(\mathbf{x}),\\\label{state3}
&&y(\mathbf{x})=0, \quad\mbox{ on }\partial \Omega,
\end{eqnarray}
where $\Omega\subset R^2$ is a bounded, polygonal domain with Lipschitz boundary separated by a closed interface $\Gamma\in C^2$, $[v]_\Gamma=(v|_{\Omega^-})|_{\Gamma}-(v|_{\Omega^+})|_{\Gamma}$, denotes the jump of the function $v(\mathbf{x})$ across an interface $\Gamma$. We call the first jump condition the Dirichlet jump condition, the second one the Neumann jump condition in (\ref{state2}). We assume that the interface $\Gamma$ separates the domain $\Omega$ into two sub-domains $\Omega^+$ and $\Omega^-$, and $\Omega^-$ lies strictly inside $\Omega$, see Figure~\ref{pic1} for an illustration. The vector $\mathbf{n}$ is the unit normal direction of $\Gamma$ pointing
to $\Omega^+$. The coefficient $\beta(\mathbf{x})$ is a positive and piecewise constant, that is,

\begin{equation}
\beta(\mathbf{x})=\left\{ \begin{array}{ll}
\beta^{-}(\mathbf{x}), ~~~&\mathbf{x}\in \Omega^-,\\
\beta^{+}(\mathbf{x}), ~~~&\mathbf{x}\in \Omega^+.\\
\end{array}
\right.
\end{equation}

\begin{figure}[!htp]
\begin{center}
       \includegraphics[width=0.3\textwidth,clip]{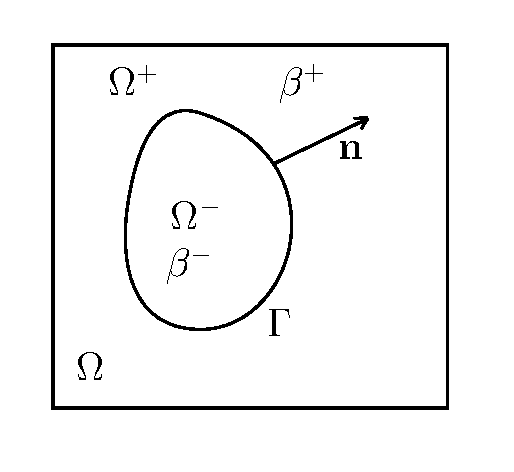}
  \caption{~The geometry of an elliptic interface problem.   }\label{pic1}
 \end{center}
\end{figure}

The weak formulation of the above state equation with interface (\ref{state1})-(\ref{state3}) can be stated as follows: find $y\in H_0^1(\Omega)$ such that

\begin{equation}\label{con_state}
a(y,v)=(f,v)+\langle u, v\rangle, ~~~~\forall v\in H^1_0(\Omega),
\end{equation}
where $a(y,v)=\sum_{s=\pm}\int_{\Omega^s}\beta^s\nabla y\cdot \nabla vd\mathbf{x}$,  $(f,v)=\int_{\Omega}fvd\mathbf{x}$, and
$\langle u, v\rangle=\int_\Gamma uvd\Gamma$. If $y=q+y_\sigma$, then this weak formulation (\ref{con_state}) is equivalent to

\begin{equation}\label{con_state1}
a(q,v)=(f,v)+\langle u, v\rangle-(\beta^{-}\nabla y_{\sigma},\nabla v)_{\Omega^{-}}, ~~~~\forall v\in H^1_0(\Omega),
\end{equation}
where $q$ is a weak solution of (\ref{state1}), (\ref{state3}) with homogeneous jump conditions in (\ref{state2}), $y_\sigma$ is a smooth function on subdomain $\Omega^-$ and satisfies homogeneous Dirichlet jump condition and nonhomogeneous Neumann jump condition on the interface{\cite{gong2008immersed}}. We can introduce the linear and continuous operator $S: L^2(\Gamma)\rightarrow H^1_0(\Omega)$ that associates an element $u\in L^2(\Gamma)$ with the unique weak solution $y\in H^1_0(\Omega)$ of (\ref{state1})-(\ref{state3}).

\begin{lemma}{\label{lm1}}
Suppose that $\Omega$ is convex, if the function $u\in H^{1/2}(\Gamma)$ and the interface $\Gamma\in C^2$. Then there exists a unique solution $y\in \widetilde{H}^2(\Omega)\cap H_0^1(\Omega)$ of problem (\ref{con_state}) such that
\begin{equation}\label{prio}
\|y\|_{\widetilde{H}^{2}(\Omega)\cap H_0^1(\Omega)}\leq C(\|f\|_{L^2(\Omega)}+\|u\|_{H^{1/2}(\Gamma)}),
\end{equation}
where $$\widetilde{H}^2(\Omega):=\left\{y\in H^1(\Omega):~y\in H^2(\Omega^s),~s=+,~-\right\},$$
equipped with the norm
$$\|y\|^2_{\widetilde{H}^2(\Omega)}:=\|y\|^2_{H^2(\Omega^+)}+\|y\|^2_{H^2(\Omega^-)}.$$
\end{lemma}
\begin{proof}
The proof is straightforward following the Lax-Milgram theorem and dual arguments\cite{brenner2008mathematical,babuvska1979direct} , see also a similar proof in \cite{huang2002some}.
\end{proof}

\subsection{Cost functional}{\label{cosfunc}}
In this paper, we will focus on the following Neumann interface control problem with PDE-constraints.
\par
$\mathbf{Problem:}$$(\mathbf{P})$ 
Consider the minimization of the cost functional
\begin{equation}\label{eq002}
J(y,u)=\frac{1}{2}\int_{\Omega}(y(\mathbf{x})-y_d(\mathbf{x}))^2d \mathbf{x}+
\frac{\alpha}{2}\int_{\Gamma}u^2(\mathbf{x})d\Gamma ,
\end{equation}
over all $(y,u)\in (\widetilde{H}^{2}(\Omega)\cap H_0^1(\Omega))\times L^2(\Gamma)$ subject to the elliptic interface problem (\ref{state1})-(\ref{state3}) and the control constraints
\begin{equation}\label{eq003}
u_a\leq u\leq u_b ~~a.e.~~on ~~\Gamma.
\end{equation}

The regularization parameter $\alpha$ is a fixed positive number and the set of admissible controls for $(\mathbf{P})$ can be written as
$$ U_{ad}=\left\{u\in L^{2}(\Gamma):u_a\leq u\leq u_b ~~a.e.~~on ~~\Gamma\right\}.$$
$\mathbf{Remark1:}$ 
The reduced formulation of problem $(\mathbf{P})$ is now given by
\begin{equation}{\label{redop}}
\min\limits_{u\in U_{ad}} \widehat{J}(u):=\frac{1}{2}\|Su-y_d\|^2_{L^2(\Omega)}+\frac{\alpha}{2}\|u\|^2_{L^2(\Gamma)},
\end{equation}
where $\widehat{J}$ is called reduced cost functional.

A control $u^*\in L^2(\Gamma)$ is called an interface optimal control or solution of problem $(\mathbf{P})$ with associated optimal state $y^*=Su^*$ if
\begin{equation}\label{opsolu}
  \widehat{J}(u^*)\leq \widehat{J}(u), ~~~~\forall u\in U_{ad}.
\end{equation}

Next we will state the existence of a solution for Problem $(\mathbf{P})$.
\begin{lemma}{\label{lm2}}
Assume that the admissible control set $U_{ad}$ is convex and closed. The optimization problem (\ref{eq002}) together with the weak formulation (\ref{con_state}) of the linear interface state equation has a unique solution $(y^*, u^*)\in {H}^{1}_0(\Omega)\times L^2(\Gamma)$.
\end{lemma}

\begin{proof}
Since the reduced functional $\widehat{J}(u)$ is bounded from below on $L^2(\Gamma)$, there exists the infimum
$$
\inf\limits_{u\in U_{ad}}\widehat{J}(u),
$$
and there is a minimizing sequence $u_n\in U_{ad}, (n=1, 2, \cdots$) such that
$$
\lim\limits_{n\rightarrow\infty}\widehat{J}(u_n)=\inf\limits_{u\in U_{ad}}\widehat{J}(u).
$$

As a reflexive Hilbert space $L^2(\Gamma)$, its bounded, closed and convex subset $U_{ad}$ is weakly sequentially compact. Consequently, $\{u_n\}, (n=1, 2, \cdots$) has a weakly convergent subsequence $\{u_{n_k}\}, (k=1, 2, \cdots$), that is $u_{n_k}\rightharpoonup u^*$ as $k\rightarrow \infty$. Again since $\widehat{J}(u)$ is strictly convex and continuous, it is weakly lower semicontinuous in $L^2(\Gamma)$. Hence, we have
$$
\widehat{J}(u^*)\leq\lim\limits_{k\rightarrow\infty}\widehat{J}(u_{n_k}).
$$

In addition, from the fact that $U_{ad}$ is closed and thus weakly closed, we have $u^*\in U_{ad}$. So we obtain
$$
\widehat{J}(u^*)=\inf\limits_{u\in U_{ad}}\widehat{J}(u),
$$
and know that such a convergence point $u^*$ is a unique global minimum of the reduced functional.
\end{proof}

We will state the Fr\'{e}chet derivative of the reduced functional $\widehat{J}(u)$ in the following lemma.
\begin{lemma}{\label{lm3}}
The Fr\'{e}chet derivative of the reduced functional $\widehat{J}(u)$ at $u\in L^2(\Gamma)$ in the direction $u-u^*\in L^2(\Gamma)$ is given by
\begin{equation}\label{frx}
\widehat{J}'(u)(u-u^*)=\langle \alpha u+p, u-u^*\rangle,
\end{equation}
where $\langle a,b\rangle=\int_\Gamma abd\Gamma$. $p=p(u)\in {H}^{1}_0(\Omega)$ is the solution of the following weak form associated adjoint state
\begin{equation}\label{adj}
(\nabla v, \beta^-\nabla p^-)+(\nabla v, \beta^+\nabla p^+)
=(Su-y_d, v),~~\forall v\in H^1_0(\Omega).
\end{equation}
\end{lemma}

$\bf{Remark2:}$ 
The adjoint state of the above corresponding weak formulation is the following interface problem
\begin{eqnarray}{\label{adjstate1}}
&&-\nabla \cdot (\beta(\mathbf{x})\nabla p(\mathbf{x}))=y(\mathbf{x})-y_d(\mathbf{x}), \quad\mbox{ in }\Omega\backslash\Gamma,\\ \label{adjstate2}
&&\left[p(\mathbf{x})\right]_{\Gamma}=0,~~\left[\beta\partial_\mathbf{n} p(\mathbf{x})\right]_{\Gamma}=0,\\ \label{adjstate3}
&&p(\mathbf{x})=0, \quad\mbox{ on }\partial \Omega.
\end{eqnarray}

Corresponding to the solution operator $S$ of the state equation, we define the adjoint operator $A^*:p=A^*(y-y_d)$ with right hand side $y-y_d$. It is well known\cite{babuvska1970finite,chen1998finite} that
\begin{equation}\label{adj-regularity}
\|p\|_{\widetilde{H}^2(\Omega)\cap H_0^1(\Omega)}\leq C(\|y\|_{L^2(\Omega)}+\|y_d\|_{L^2(\Omega)}).
\end{equation}

\begin{proof}
We construct the Lagrangian functional $L(y, u, p)=J(y, u)+(f, p)+\langle u, p \rangle-a(y, p)$, then with $y=y(u)=Su$, we have
$\widehat{J}(u)=J(y(u), u)=L(y(u), u, p)$. For any $u\in L^2(\Gamma)$, it is shown that
\begin{eqnarray}{\label{lag}}
\widehat{J}'(u)(u-u^*)=L'_y(y(u),u, p)(y'_u(u)(u-u^*))+L'_u(y(u),u,p)(u-u^*) \nonumber \\
+L'_p(y(u),u,p)(p'_u(u-u^*)).
\end{eqnarray}

Since by construction $L'_p(y(u),u,p)(\cdot)=0$ and we choose $p=p(u)$ such that $L'_y(y(u),u, p)=0$, this is nothing else but the adjoint state (\ref{adj}), then we have
$$
\widehat{J}'(u)(u-u^*)=L'_u(y(u),u,p)(u-u^*)=\langle \alpha u+ p, u-u^*\rangle, ~~\forall u\in L^2(\Gamma).
$$
\end{proof}

\subsection{Optimality system}
We make the following smoothness assumption on the data of the problem, that is, $y_d\in C^{0,\lambda}(\overline{\Omega})\subset L^2(\Omega)$ with some $\lambda\in (0,1)$, and control bounds $u_a<u_b$ are fixed real numbers. As a consequence of the above lemmas, by using standard techniques \cite{pinnau2008optimization,troltzsch2010optimal}, we have the existence of solutions and the optimality conditions.

\begin{theorem}{\label{th}}
Let $(y^*, u^*)\in (\widetilde{H}^{2}(\Omega)\cap H_0^1(\Omega))\times U_{ad}$ be the solution of interface control problem ($\mathbf{P}$) with PDE-constraints, then there exists an adjoint state $p^*\in H^1_0(\Omega)$ such that the triplet $(y^*, u^*, p^*)$ satisfies the following optimality system
\begin{eqnarray}{\label{opts}}
& a(y^*,v)=(f,v)+\langle u^*, v\rangle, & ~~\forall v\in {H}^{1}_0(\Omega),\\ \label{opts2}
& a(v,p^*)=(Su^*-y_d, v), & ~~\forall   v\in {H}^{1}_0(\Omega), \\ \label{opts3}
& \langle \alpha u^*+ p^*, u-u^*\rangle\geq 0, & \forall u\in U_{ad}.
\end{eqnarray}

Moreover, the variational inequality is equivalent to the following projection equation
\begin{equation}\label{proj_equa}
u^*=\mathcal{P}_{[u_a, u_b]}\left(-\frac{1}{\alpha}p^*|_\Gamma\right),
\end{equation}
here, $\mathcal{P}_{[u_a, u_b]}(u)=min\{u_b, max\{u_a, u \}\}$ denotes the pointwise projection onto $[u_a, u_b]$.
\end{theorem}

\begin{proof}
Let $(y^*, u^*)\in (\widetilde{H}^{2}(\Omega)\cap H_0^1(\Omega))\times U_{ad}$ be a solution of the interface optimal control problem ($\mathbf{P}$), then by the definition of $y^*$ and $p^*$, it is obvious that (\ref{opts}) and (\ref{opts2}) hold.
 Since $u^*\in L^2(\Gamma)$ is a solution of the reduced functional $\widehat{J}(u)$, then for any $\varepsilon>0$, we have
$$
\frac{1}{\varepsilon}(\widehat{J}(u^*+\varepsilon (u-u^*))-\widehat{J}(u^*))\geq 0, ~~\forall u\in U_{ad}.
$$

Then letting $\varepsilon\rightarrow 0$ yields $\widehat{J}'(u^*)(u-u^*)\geq 0$, and using Lemma~\ref{lm3}, the variational inequality (\ref{opts3}) holds.
\end{proof}

Using the above Lemma~\ref{lm1}, we have the following regularity result.
\begin{theorem}
Let $\left(y^*,u^*,p^*\right)$ be the solution of the problem $(\mathbf{P})$. Then we have
$$(y^*,u^*,p^*)\in (\widetilde{H}^{2}(\Omega)\cap {H}^{1}_0(\Omega))\times U_{ad}\times (\widetilde{H}^{2}(\Omega)\cap {H}^{1}_0(\Omega)).$$
\end{theorem}
\begin{proof}
Since $f\in L^2(\Omega)$, and $u^*\in L^2(\Gamma)$, we have $y^*\in H^1_0(\Omega)$. From $y^*-y_d \in L^2(\Omega)$ and elliptic regularity, it implies that $p^*\in \widetilde{H}^{2}(\Omega)\cap {H}^{1}_0(\Omega)$, using the trace theorem with this together, and in view of equation (\ref{proj_equa}), we obtain $u^*\in H^1(\Gamma)$, which in turns gives $y^*\in \widetilde{H}^{2}(\Omega)\cap {H}^{1}_0(\Omega)$.
\end{proof}

\section{Numerical approximation}\label{num}
Since the analytic solution to problem $(\mathbf{P})$ is rarely available, we seek numerical solutions. In order to get an accurate solution  without using a body fitted mesh to resolve the interface problem, we use the immersed finite element method (IFEM) which is described in the next subsection.

\subsection{The immersed finite element method}\label{section3_1}
The solution of interface problem (\ref{state1})-(\ref{state3}) belongs to $\widetilde{H}^2(\Omega)$ because of the discontinuity in the coefficient\cite{babuvska1970finite,bramble1996finite,huang2002some}, see Lemma~\ref{lm1}.
The standard finite element methods can not achieve the optimal convergence unless
the mesh fits with the interface. The immersed finite element can achieve optimal convergence using a  uniform mesh.

Let $\mathcal{T}_h$ be a uniform triangulation of $\Omega$ with mesh size $h$. If $h$ is sufficiently small, then it is reasonable to assume the following:
\begin{itemize}
  \item  The interface $\Gamma$ will not intersect an edge of any element at more than two points unless the edge is part of $\Gamma$;
  \item If $\Gamma$ intersects the boundary of an element at two points, then these two points must be on different edges of this element.
\end{itemize}

As usual, the triangulation $\mathcal{T}_h$  defines in a natural way a segmentation of interface $\Gamma$. The interface $\Gamma$ is approximated by $\Gamma_h$, the union of $\Gamma_h^i$, the line segments connecting the intersections of the interface and the edges of elements.

We call an element $T$ an interface element if $\Gamma_h$ passes through the interior of the element $T$; otherwise we call $T$ a non-interface element. Note that
if $\Gamma$ intersects $\partial T$ at two vertices of $T$, then the element is a non-interface element.

The sets of all interface elements and non-interface elements are denoted by $\mathcal{T}_h^{int}$ and $\mathcal{T}_h^{non}$, respectively.

\begin{figure}[!htp]
\begin{center}
       \includegraphics[width=0.3\textwidth,clip]{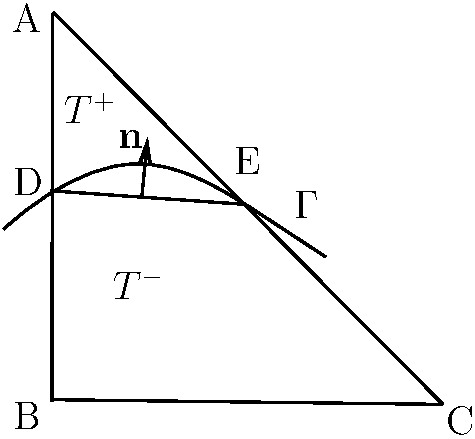}
  \caption{~A typical interface element and a neighboring element.   }\label{pic2}
 \end{center}
\end{figure}

The idea of the IFEM is to locally modify the standard linear FE basis functions in interface elements to take into account the interface jump conditions (\ref{state2}). For simplicity, we consider a typical interface element $\triangle ABC$ whose geometric configuration is given in Figure~\ref{pic2} as a demonstration. The line segment $\overline{DE}$ divides $T$ into two parts $T^+$ and $T^-$ and $\mathbf{n}$ is the unit normal direction of $\overline{DE}$.  We construct the following piecewise linear function
\begin{equation} \label{ifembase}
  \phi(x)=\left\{ \begin{array}{ll}
  \phi^+=a^++b^+x_1+c^+x_2,\qquad &\mathbf{x}=(x_1,x_2)\in T^+,\\
 \phi^-=a^-+b^-x_1+c^-x_2,\qquad &\mathbf{x}=(x_1,x_2)\in T^-,\\
\end{array}  \right.
\end{equation}
where the coefficients are chosen such that
\begin{equation}\label{ifembase_vi}
\begin{array}{lll}
\phi(A)=V_1,\quad & \phi(B)=V_2,\quad & \phi(C)=V_3,\\
\phi^+(D)=\phi^-(D),~ & \phi^+(E)=\phi^-(E),~& \beta^+\partial_\mathbf{n}\phi^+=\beta^-\partial_\mathbf{n}\phi^-,\\ 
\end{array}
\end{equation}
where $V_i$, $i=1,2,3$ are the nodal variables. Notice that there are six parameters in (\ref{ifembase}) and six constraints in (\ref{ifembase_vi}). The piecewise linear function is uniquely determined by $V_i$, $i=1,2,3$, see \cite{li2003new}.  We note that a function is likely discontinuous across the common edges of two adjacent interface elements, see \cite{li2003new} for the detailed description of the phenomenon.

\begin{figure}[!htp]
\begin{center}
       \includegraphics[width=0.8\textwidth,clip]{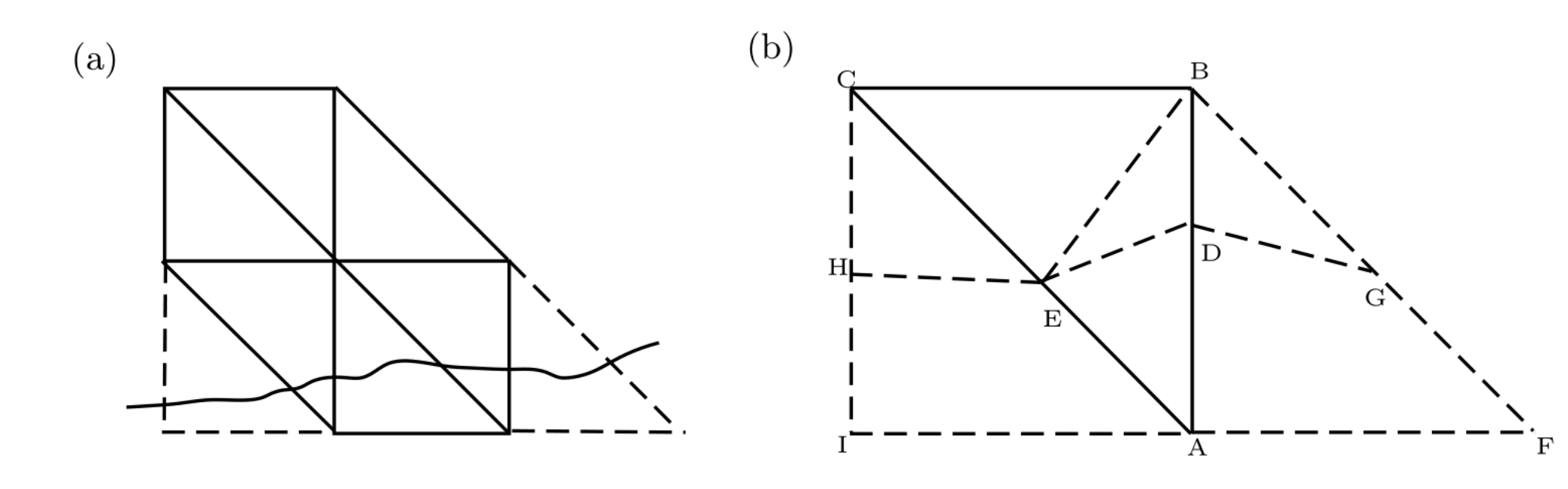}
 \caption{~$(a)\,$ The support of a local basis function.
 $(b)\,$ A diagram for the construction of a local basis function on $\triangle ABC $.}
 \label{conform1}
 \end{center}
\end{figure}

A conforming element was proposed in \cite{li2003new}. The idea of the IFEM is to make sure that some of the local basis functions in two adjacent interface elements can take the same value at the interface point on their common edge. For the non-interface element, basis function $\phi_i$ is the usual conforming linear finite element basis function; for the interface element, to construct a local basis function that is globally continuous, we extend the previously defined basis function at the same node to one more triangle along the interface (cf. Figure 3). We assign values of a local basis function $\phi_i$ at the vertices $A$, $B$, $C$, $F$ and $I$, and this construction consists of the following five steps.

S1. The three non-conforming IFE functions defined as above on the elements $\triangle ABC$, $\triangle AFB$ and $\triangle ACI$ are formed by using the values at the vertices $A$, $B$, $C$, $F$ and $I$, respectively.

S2. The value at $D$ is taken as the average of the values at $D$ of the nonconforming piecewise linear basis functions defined on $\triangle ABC$ and $\triangle AFB$ constructed in S1.

S3. Similarly, the value at $E$ is taken as the average of values at $E$ of the nonconforming piecewise linear basis functions defined on the elements $\triangle ABC$ and $\triangle ACI$ formed in S1.

S4. Using an auxiliary line, say, line segment $\overline{BE}$ or $\overline{DC}$ to partition the element $\triangle ABC$ into three subtriangles. The auxiliary line is chosen in such a way that at least one of angles(or complimentary angles if the angle is more than $\pi /4$) is bigger than or equal to $\pi /2$.

S5. The piecewise linear basis functions in the three subtriangles are determined by the values at the points $A$, $B$, $C$, $D$ and $E$.

It has been proved in \cite{gong2008immersed} that the IFE space $H_h(\Omega)$ has an optimal approximation capability, where $H_h(\Omega)$ is the span of the basic functions $\{\phi_i\}_{i=1}^N$, where $\phi_i\in C(\overline{\Omega})\cap H^1_0(\Omega)$. Each function in $H_h(\Omega)$ satisfies the approximate homogeneous jump conditions. In other words, $H_h(\Omega)\subset H_0^1(\Omega)$. For a small positive constant $\rho>0$, we denote the neighborhood of $\Gamma$ by
$$N\left(\Gamma,\rho\right)=\left\{x\in\Omega:\textup{dist}\left(x,\Gamma\right)\leq\rho\right\},$$
where $dist\left(x,\Gamma\right)$ is the distance from $x$ to $\Gamma$. We define $\Omega_\Gamma=\bigcup\limits_{T\in\mathcal{T}^{int}_h}T.$ And note that $\Omega_\Gamma\subset N\left(\Gamma,h\right),$ we assume that the mesh size $h<\rho$ so that $\Omega_\Gamma\subset N\left(\Gamma,\rho\right).$ Since the partition $\mathcal{T}_h$ is regular and quasi-uniform, it is true that $|\Omega_\Gamma|\leq Ch.$ Then, we define $u^e(x)=u^e(x^*+\delta \mathbf{n})=u(x^*),~ \forall x\in N\left(\Gamma,\rho\right)$ to be the extensions of $u$, where $x^*$ is the projection of $x$ on the interface $\Gamma$ and $|\delta|\leq \rho.$ Due to the interface $\Gamma$ is smooth and $\rho$ is a small positive constant, for any $x\in N\left(\Gamma,\rho\right)$, there exits a unique orthogonal projection $x^*\in \Gamma$, namely, $\|x-x^*\|_2=\textup{dist}\left(x,\Gamma\right)$. We define extension function $y_{\sigma u}: \Omega\rightarrow R$ by
\begin{equation*}
  y_{\sigma u}=\left\{ \begin{aligned}
  &H\left(\varphi\left(x\right)\right)I_h\tilde{y}_u\left(x\right)-I_h\left(H\left(\varphi\left(x\right)\right)\tilde{y}_u\left(x\right)\right),\quad &&\textup{if}\ ~ x\in\Omega_\Gamma,\\
 &0,\quad &&\textup{otherwise},\\
\end{aligned}  \right.
\end{equation*}
where $\displaystyle{\tilde{y}_u\left(x\right)=\frac{u^e(x)}{\beta^-}d\left(x\right),}$
\begin{equation*}
 d\left(x\right)=\left\{ \begin{aligned}
  &-\textup{dist}\left(x,\Gamma\right),\quad && x\in\Omega^-\cap N\left(\Gamma,\rho\right),\\
 &0,\quad && x\in\Gamma,\\
 &\textup{dist}\left(x,\Gamma\right),\quad && x\in\Omega^+\cap N\left(\Gamma,\rho\right),
\end{aligned}  \right.
 \qquad H\left(s\right)=\left\{ \begin{aligned}
  &0,\quad && s\geq0,\\
 &1,\quad && s<0,
\end{aligned}  \right.
\end{equation*}
$I_h$ is the linear interpolation operator and $\varphi(x)$ is a typical level set function that represents the interface, see \cite{gong2008immersed, ji2018high} for details. For convenience, we define $\widetilde{H}_h(\Omega)$ as following:
$$
\widetilde{H}_h(\Omega)=\{ \psi_h\mid \psi_h=q_h+y_{\sigma u}, ~ q_h\in H_h(\Omega) ~ \mbox{and} ~ \left[\beta\partial_ny_{\sigma u}\right]_{\Gamma}=u\}.
$$
\begin{lemma}\label{lm-sm}
For any $u\in L^2(\Gamma)$, find $y_h(u)=q_h\left(u\right)+y_{\sigma u}\in \widetilde{H}_h(\Omega)$ such that
\begin{equation}\label{IFEM}
a_h(q_h(u),v_h)=( f,v_h)+\langle u, v_h\rangle-\sum_{T\in\mathcal{T}_h^{int}}\sum_{s=\pm}\int_{T\cap\Omega^{s}}\beta\nabla y_{\sigma u}\nabla v_h dx, \quad\forall v_h\in H_h(\Omega),
\end{equation}
where
$$a_h(q_h(u),v_h):=\sum_{T\in\mathcal{T}_h}\int_{T}\beta\nabla q_h(u)\cdot \nabla v_h dx, \quad \forall v_h  \in H_h(\Omega).$$
\end{lemma}

Note that the above semi-discrete scheme is well defined and admits a unique solution $y_h(u)=q_h\left(u\right)+y_{\sigma u},\  q_h\left(u\right)\in H_h(\Omega)$, which we denote by $y_h(u)=S_h(u)$. Since $y(u)=q+y_{\sigma u}$, $y_h(u)=q_h\left(u\right)+y_{\sigma u}$, and note that $\|y-y_h(u)\|_{H^1(\Omega)}\leq \|q-q_h\left(u\right)\|_{H^1(\Omega)}$, $\|y-y_h(u))\|_{L^2(\Omega)}\leq \|q-q_h\left(u\right)\|_{L^2(\Omega)}$, the proof of the following Theorem is standard and can be proved by dual argument\cite{brenner2008mathematical,babuvska1979direct}, see\cite{gong2008immersed, ji2018high}.
\begin{theorem}\label{theo_IFEM}
Let $y\in (\widetilde{H}^{2}(\Omega)\cap {H}^{1}_0(\Omega))$ and $y_h(u)=q_h\left(u\right)+y_{\sigma u}\in \widetilde{H}_h(\Omega)$ be the solutions of (\ref{state1})-(\ref{state3}) and (\ref{IFEM}) respectively. Then there exists a constant $C>0$ such that
\begin{eqnarray}\label{h1}
&\|y-y_h(u)\|_{H^1(\Omega)}\leq Ch(\|f\|_{L^2(\Omega)}+\|u\|_{H^{1/2}(\Gamma)}+\|\tilde{y}_u\|_{H^3\left(N\left(\Gamma,\rho\right)\right)}),\\   \label{l2}
&\|y-y_h(u)\|_{L^2(\Omega)}\leq Ch^2(\|f\|_{L^2(\Omega)}+\|u\|_{H^{1/2}(\Gamma)}+\|\tilde{y}_u\|_{H^3\left(N\left(\Gamma,\rho\right)\right)}).
\end{eqnarray}
\end{theorem}

\subsection{Discretization of the optimal control problem}\label{section3_2}
We will state the discrete version of the state equation. For each $u\in L^2(\Gamma)$, find the unique element $y_h(u)=S_hu=q_h\left(u\right)+y_{\sigma u}\in \widetilde{H}_h(\Omega)$ satisfying
\begin{equation}\label{semeq}
a_h(q_h(u),v_h)=( f,v_h)+\langle u, v_h\rangle-\sum_{T\in\mathcal{T}_h^{int}}\sum_{s=\pm}\int_{T\cap\Omega^{s}}\beta\nabla y_{\sigma u}\nabla v_h dx, \quad\forall v_h\in H_h(\Omega).
\end{equation}

We can define the immersed finite element semidiscrete approximation of the optimal control problem with the variational discretization concept~\cite{hinze2005variational} in the following
\par
$\mathbf{Problem:}$ $(\mathbf{\tilde{P}_h})$ 
Consider the problem of minimizing in $U_{ad}$
\begin{equation}\label{sem-dis_J}
J_h(y_h(u), u)=\frac{1}{2}\int_\Omega(y_h(u)-y_d)^2d\mathbf{x}+\frac{\alpha}{2}\int_\Gamma u^2d\Gamma,
\end{equation}
over all $y_h\left(u\right)=q_h\left(u\right)+y_{\sigma u},\ (q_h(u),u)\in H_h(\Omega)\times U_{ad}$ subject to
\begin{equation}\label{dis_state_equation}
a_h(q_h(u),v_h)=( f,v_h)+\langle u, v_h\rangle-\sum_{T\in\mathcal{T}_h^{int}}\sum_{s=\pm}\int_{T\cap\Omega^{s}}\beta\nabla y_{\sigma u}\nabla v_h dx, \quad\forall v_h\in H_h(\Omega),
\end{equation}
where $S_h:y_h(u)=S_hu=q_h\left(u\right)+y_{\sigma u}\in \widetilde{H}_h(\Omega)$ is the discrete solution operator.

The adjoint of the discrete solution operator is equal to the discretized version of the adjoint solution operator. Thus we can write $A_h^*:p_h(u)=A_h^*(S_hu-y_d)$, the discrete adjoint state is the unique element $p_h(u)\in  H_h(\Omega)$ such that
\begin{equation}\label{dis-adj}
a_h(v_h, p_h(u))=(S_hu-y_d, v_h), ~~\forall v_h\in  H_h(\Omega).
\end{equation}

The following result for the adjoint state holds analogously with Theorem~\ref{theo_IFEM}.

\begin{theorem}\label{theo_IFEM-adj}
Let $p\in (\widetilde{H}^{2}(\Omega)\cap {H}^{1}_0(\Omega))$ and $p_h(u)\in  H_h(\Omega)$ be the solutions of (\ref{adjstate1})-(\ref{adjstate3}) and (\ref{dis-adj}) respectively. Then there exists a constant $C>0$ such that
\begin{eqnarray}\label{ph1}
&\|p-p_h(u)\|_{L^2(\Omega)}+h\|p-p_h(u)\|_{H^1(\Omega)}\\ \nonumber
&\leq Ch^2(\|y_d\|_{L^2(\Omega)}+\|f\|_{L^2(\Omega)}+\|u\|_{H^{1/2}(\Gamma)}+\|\tilde{y}_u\|_{H^3\left(N\left(\Gamma,\rho\right)\right)}),\\ \nonumber
&\|p-p_h(u)\|_{L^{2}(\Gamma)}\\ \label{p-gamal2}
&\leq Ch^{3/2}(\|y_d\|_{L^2(\Omega)}+\|f\|_{L^2(\Omega)}+\|u\|_{H^{1/2}(\Gamma)}+\|\tilde{y}_u\|_{H^3\left(N\left(\Gamma,\rho\right)\right)}).
\end{eqnarray}
\end{theorem}

\begin{proof}
For (\ref{ph1}), we can use the similar method which proves Theorem~\ref{theo_IFEM} in \cite{gong2008immersed} to deduce (\ref{ph1}). From coercivity, for all $v_h\in H_h(\Omega)$, we have
\begin{equation*}
\begin{aligned}
&\alpha\|p-p_h(u)\|^2_{H^1(\Omega)}\\
&\quad\leq a(p-p_h(u), p-p_h(u))=a(p-p_h(u), p-v_h)+a(p-p_h(u), v_h-p_h(u))\\
&\quad=a(p-p_h(u), p-v_h)+(y-y_h(u), v_h-p_h(u))\\
&\quad=a(p-p_h(u), p-v_h)+(y-y_h(u), v_h-p)+(y-y_h(u), p-p_h(u))\\
&\quad\leq Ch\|p-p_h(u)\|_{H^1(\Omega)}\|p\|_{\widetilde{H}^2(\Omega)}+Ch^2\|y-y_h(u)\|_{L^2(\Omega)}\|p\|_{\widetilde{H}^2(\Omega)}
\\
&\quad\quad+\|y-y_h(u)\|_{L^2(\Omega)}\|p-p_h(u)\|_{L^2(\Omega)}.
\end{aligned}
\end{equation*}

Using a Poincar$\acute{\mbox{e}}$ inequality, Theorem~\ref{theo_IFEM}, Young's inequality and elliptic regularity, we obtain
\begin{equation}\label{p-h}
\|p-p_h(u)\|_{H^1(\Omega)}\leq Ch(\|f\|_{L^2(\Omega)}+\|u\|_{H^{1/2}(\Gamma)}+\|\tilde{y}_u\|_{H^3\left(N\left(\Gamma,\rho\right)\right)}+\|y_d\|_{L^2(\Omega)}).
\end{equation}

For the error estimate $\|p-p_h(u)\|_{L^2(\Omega)}$, we will use a duality argument from a interface problem with homogeneous jump conditions, let $\varphi\in \widetilde{H}^{2}(\Omega)\cap {H}^{1}_0(\Omega)$ be the solution of the interface problem with homogeneous jump conditions, we have
$$
a(v, \varphi)=(v, p-p_h(u)), \qquad\forall v\in H^1_0(\Omega).
$$

Therefore, for all $v_h\in H_h(\Omega)$
\begin{equation*}
\begin{aligned}
&\|p-p_h(u)\|^2_{L^2(\Omega)}\\
&\quad=(p-p_h(u), p-p_h(u))=a(p-p_h(u), \varphi)\\
&\quad=a(p-p_h(u), \varphi-v_h)+a(p-p_h(u), v_h)\\
&\quad=a(p-p_h(u), \varphi-v_h)+(y-y_h(u), v_h-\varphi)+(y-y_h(u), \varphi)\\
&\quad\leq Ch\|p-p_h(u)\|_{H^1(\Omega)}\|\varphi\|_{\widetilde{H}^{2}(\Omega)}
+Ch^2\|y-y_h(u)\|_{L^2(\Omega)}\|\varphi\|_{\widetilde{H}^{2}(\Omega)}
\\
&\quad\quad+\|y-y_h(u)\|_{L^2(\Omega)}\|\varphi\|_{{L}^{2}(\Omega)}.
\end{aligned}
\end{equation*}

The error estimate (\ref{ph1}) can be obtained by using (\ref{p-h}), a Poincar$\acute{\mbox{e}}$ inequality, Theorem~\ref{theo_IFEM}, Young's inequality and elliptic regularity.

Next we will prove (\ref{p-gamal2}), we will use the duality argument, and consider the following auxiliary interface problem with inhomogeneous flux jump condition: given $g=\left(p-p_h(u)\right)\big|_{\Gamma}\in L^{2}(\Gamma)$, find $\varphi_g=\hat{\varphi}_{g}+\varphi_{g\sigma}\in \widetilde{H}^{3/2}_0(\Omega)={H}^{1}_0(\Omega)\cap {H}^{3/2}(\Omega^{\pm})$ such that
\begin{eqnarray}{\label{auxiliary}}
&&-\nabla \cdot (\beta(\mathbf{x})\nabla\varphi_g(\mathbf{x}))=0, \quad\mbox{ in }\Omega\backslash\Gamma,\\ \label{auxiliary2}
&&\left[\varphi_g(\mathbf{x})\right]_{\Gamma}=0,~~\left[\beta\partial_\mathbf{n} \varphi_g(\mathbf{x})\right]_{\Gamma}=g,\\
\label{auxiliary3}
&&\varphi_g(\mathbf{x})=0, \quad\mbox{ on }\partial \Omega.
\end{eqnarray}

Let $\varphi_{g_h}=\hat{\varphi}_{g_h}+\varphi_{g\sigma},~ \hat{\varphi}_{g_h}\in H_{h}(\Omega)$ be the solution of the corresponding variational problem
\begin{equation}\label{dis_state_auxiliary}
a_h(v_h,\hat{\varphi}_{g_h})=\langle v_h, g\rangle-\sum_{T\in\mathcal{T}_h^{int}}\sum_{s=\pm}\int_{T\cap\Omega^{s}}\beta\nabla \varphi_{g\sigma}\nabla v_h dx, \qquad\forall v_h\in H_h(\Omega),
\end{equation}
where the extension function $\varphi_{g\sigma}: \Omega\rightarrow R$ by
\begin{equation*}
  \varphi_{g\sigma}=\left\{ \begin{aligned}
  &H\left(\varphi\left(x\right)\right)I_h\tilde{\varphi}_g\left(x\right)-I_h\left(H\left(\varphi\left(x\right)\right)\tilde{\varphi}_g\left(x\right)\right),\quad &&\textup{if}\ ~ x\in\Omega_\Gamma,\\
 &0,\quad &&\textup{otherwise},\\
\end{aligned}  \right.
\end{equation*}
where $\displaystyle{\tilde{\varphi}_g\left(x\right)=\frac{g^e(x)}{\beta^-}d\left(x\right)}$ and $g^e(x)=g^e(x^*+\delta \mathbf{n})=g(x^*),~ \forall x\in N\left(\Gamma,\rho\right), ~ x^*\in \Gamma$ to be the extensions of $g$.

Then we have
\begin{equation*}
\begin{aligned}
&\langle p-p_h(u), g\rangle\\
&\quad=a_h(p-p_h(u), \hat{\varphi}_g)-\sum_{T\in\mathcal{T}_h^{int}}\sum_{s=\pm}\int_{T\cap\Omega^{s}}\beta\nabla \varphi_{g\sigma}\nabla \left(p-p_h(u)\right) dx\\
&\quad=a_h(p-p_h(u), \hat{\varphi}_g-\hat{\varphi}_{g_h})+a_h(p-p_h(u),\hat{\varphi}_{g_h})\\
&\quad\quad-\sum_{T\in\mathcal{T}_h^{int}}\sum_{s=\pm}\int_{T\cap\Omega^{s}}\beta\nabla \varphi_{g\sigma}\nabla \left(p-p_h(u)\right) dx\\
&\quad=a_h(p-p_h(u), \varphi_g-\varphi_{g_h})+a_h(p-p_h(u), \hat{\varphi}_{g_h})\\
&\quad\quad-\sum_{T\in\mathcal{T}_h^{int}}\sum_{s=\pm}\int_{T\cap\Omega^{s}}\beta\nabla \varphi_{g\sigma}\nabla \left(p-p_h(u)\right) dx\\
&\quad=K_1+K_2+K_3.
\end{aligned}
\end{equation*}

For $K_1$, and note that $\varphi_{g_h}-\varphi_g\in H^1_0(\Omega)$, we have
$$
|K_1|\leq \|p-p_h(u)\|_{H^1(\Omega)}\|\varphi_{g_h}-\varphi_g\|_{H^1(\Omega)}\leq Ch^{3/2}\|p\|_{\widetilde{H}^{2}(\Omega)}\|g\|_{L^2(\Gamma)},
$$
where we use $\|p-p_h(u)\|_{{H}^{1}(\Omega)}\leq Ch\|p\|_{\widetilde{H}^2(\Omega)}$(similarly with Theorem 6, \cite{Hansbo2002An}) and the fact $\|\varphi_{g_h}-\varphi_g\|_{H^1(\Omega)}\leq Ch^{1/2}\|g\|_{L^2(\Gamma)}$ (similarly with (4.4) for $s=1$ in Theorem 4.1,\cite{Casas2008Error}).

For $K_2$, we obtain
\begin{equation*}
\begin{aligned}
|K_2|&=|a_h(p-p_h(u), \hat{\varphi}_{g_h})|\\
&=|(y-y_h(u), \varphi_{g_h}-\varphi_g)+(y-y_h(u), \varphi_g-\varphi_{g\sigma})|\\
&\leq\|y-y_h(u)\|_{{L}^{2}(\Omega)}\|\varphi_{g_h}-\varphi_{g}\|_{{L}^{2}(\Omega)}+\|y-y_h(u)\|_{{L}^{2}(\Omega)}\left(\|\varphi_{g}\|_{{L}^{2}(\Omega)}+\|\varphi_{g\sigma}\|_{{L}^{2}(\Omega)}\right)\\
&\leq Ch^2(\|f\|_{L^2(\Omega)}+\|u\|_{H^{1/2}(\Gamma)}+\|\tilde{y}_u\|_{H^3\left(N\left(\Gamma,\rho\right)\right)})\cdot Ch^{3/2}\|g\|_{L^2(\Gamma)}\\
&~~+ Ch^2(\|f\|_{L^2(\Omega)}+\|u\|_{H^{1/2}(\Gamma)}+\|\tilde{y}_u\|_{H^3\left(N\left(\Gamma,\rho\right)\right)})\cdot \left(C\|\varphi_{g}\|_{\widetilde{H}^{3/2}_0(\Omega)}+\|\varphi_{g\sigma}\|_{{L}^{2}(\Omega_\Gamma)}\right),
\end{aligned}
\end{equation*}
where we use the Theorem~\ref{theo_IFEM}, the fact $\|\varphi_{g_h}-\varphi_g\|_{L^2(\Omega)}\leq Ch^{3/2}\|g\|_{L^2(\Gamma)}$ (similarly with (4.4) for $s=0$ in Theorem 4.1,\cite{Casas2008Error}), the embedding theorem $H^1(\Omega)\hookrightarrow L^\infty(\Omega)$, use the fact $g=\left(p-p_h(u)\right)\big|_{\Gamma}$ and the following estimate
\begin{equation}\label{extension}
\begin{aligned}
\|\varphi_{g\sigma}\|_{{L}^{2}(\Omega_\Gamma)}&=\left(\sum_{T\in\mathcal{T}_h^{int}}\|H\left(\varphi\left(x\right)\right)I_h\left(\tilde{\varphi}_{g}\right)-I_h\left(H\left(\varphi\left(x\right)\right)\tilde{\varphi}_{g} \right)\|^2_{L^{2}(T)}\right)^{1/2}\\
&\leq Ch\left(\sum_{T\in\mathcal{T}_h^{int}}\|
\tilde{\varphi}_{g}\|^2_{L^{\infty}(T)}\right)^{1/2}= Ch\left(\sum_{T\in\mathcal{T}_h^{int}}\|
\frac{g^e(x)}{\beta^-}d\left(x\right)\|^2_{L^{\infty}(T)}\right)^{1/2}\\
&\leq Ch^{2}\left(\sum_{T\in\mathcal{T}_h^{int}}\|g\|^2_{L^{\infty}(\tilde{\Gamma}_T)}\right)^{1/2}\leq Ch^{2}\left(\sum_{T\in\mathcal{T}_h^{int}}\|g\|^2_{L^{\infty}(T)}\right)^{1/2}\\
&\leq Ch^{2}\left(\sum_{T\in\mathcal{T}_h^{int}}\|g\|^2_{H^{1}(T)}\right)^{1/2}\leq Ch^{2}\|p-p_h(u)\|_{{H}^{1}(\Omega)},
\end{aligned}
\end{equation}
where $\tilde{\Gamma}_T$ is the projection of $T$ onto $\Gamma$.

For $K_3$, similar to $\left(\ref{extension}\right)$ we have $\|\varphi_{g\sigma}\|_{{H}^{1}(\Omega_\Gamma)}\leq Ch\|p-p_h(u)\|_{{H}^{1}(\Omega)}$ and using the definition of the extension function, we obtain
\begin{equation*}
\begin{aligned}
|K_3|&=|-\sum_{T\in\mathcal{T}_h^{int}}\sum_{s=\pm}\int_{T\cap\Omega^{s}}\beta\nabla \varphi_{g\sigma}\nabla \left(p-p_h(u)\right) dx|\\
&\leq C\|p-p_h(u)\|_{{H}^{1}(\Omega)}\|\varphi_{g\sigma}\|_{{H}^{1}(\Omega_\Gamma)}\\
&\leq Ch\|p-p_h(u)\|^2_{{H}^{1}(\Omega)}.
\end{aligned}
\end{equation*}

Since $\|\varphi_g\|_{\widetilde{H}^{3/2}_0(\Omega)}\leq C\|g\|_{L^{2}(\Gamma)}=C\|p-p_h(u)\|_{L^{2}(\Gamma)}$,
$\|p\|_{\widetilde{H}^2(\Omega)}\leq C(\|y\|_{L^2(\Omega)}+\|y_d\|_{L^2(\Omega)})$, using Young's inequality and $\left(\ref{p-h}\right)$, we get
$$
\|p-p_h(u)\|_{L^{2}(\Gamma)}\leq Ch^{3/2}(\|y_d\|_{L^2(\Omega)}+\|f\|_{L^2(\Omega)}+\|u\|_{H^{1/2}(\Gamma)}+\|\tilde{y}_u\|_{H^3\left(N\left(\Gamma,\rho\right)\right)}).
$$
\end{proof}

The finite dimensional approximation of the optimal control problem in the full discretization form reads as follows:
\par
$\mathbf{Problem:}$ $(\mathbf{P_h})$ 
Consider the problem of minimizing in $U_{ad}^h$
\begin{equation}\label{full-dis_J}
J_h(y_h,u_h)=\frac{1}{2}\int_\Omega(y_h-y_d)^2d\mathbf{x}+\frac{\alpha}{2}\int_{\Gamma_h} u_h^2d\Gamma_h,
\end{equation}
over all $y_h=q_h+y_{\sigma u},\  (q_h,u_h)\in  H_h(\Omega)\times U_{ad}^h$ subject to
\begin{equation}\label{full-dis_state_equation}
a_h(q_h,v_h)=(f,v_h)+\langle u_h, v_h\rangle_{\Gamma_h}-\sum_{T\in\mathcal{T}_h^{int}}\sum_{s=\pm}\int_{T\cap\Omega^{s}}\beta\nabla y_{\sigma u}\nabla v_h dx, \quad\forall v_h\in H_h(\Omega),
\end{equation}
where $U_{ad}^h=\{u_h\in \overline{U}_h: u_a\leq u_h\leq u_b ~~a.e.~~on ~~\Gamma_{h}\}$, $\overline{U}_h=\{u_h\in L^\infty (\Gamma_{h}): u_h|_{\Gamma_h^i} =constant \}$, $U_h=\{u_h\in L^\infty (\Gamma): u_h|_{\Gamma^i} =constant \}$, $S_h:y_h=S_h\hat{u}_h$ is the full discrete solution operator, where $\hat{u}_h=u_h\circ g_h^{-1}\in U_h\cap U_{ad}$ and the definition of $g_h^{-1}$ will be given later.

$\mathbf{Remark3:}$ 
The discrete optimal control problem $(\mathbf{P_h})$ admits a unique solution $u^*_h$. The optimal discrete state $S_hu^*_h$ is denoted by $y_h^*$ and the optimal discrete adjoint state $p_h^*$ is given by $p_h^*=A_h^*(S_hu_h^*-y_d)$.

Similar to Theorem~\ref{th}, we have the following Theorem for $(\mathbf{P_h})$.
\begin{theorem}{\label{th2}}
Let $y^*_h=q_h^*+y_{\sigma u^*},\  (q_h^*, u^*_h)\in  H_h(\Omega)\times U_{ad}^h$ be the solution of interface control problem $(\mathbf{P_h})$ with PDE-constraints , then there exists an adjoint state $p^*_h\in  H_h(\Omega)$ such that the triplet $(y^*_h, u^*_h, p^*_h)$ satisfies the following optimality system
\begin{eqnarray}{\label{dopts}}
& a_h(y_h^*,v_h)=(f,v_h)+\langle u_h^*, v_h\rangle_{\Gamma_h}, & ~~~~\forall v_h\in H_h(\Omega),\\ \label{dopts2}
& a_h(v_h,p_h^*)=(S_{h}\hat{u}^*_h-y_d, v_h ), & ~~~~\forall  v_h\in  H_h(\Omega), \\ \label{dopts3}
& \langle \alpha u_h^*+ p_h^*, u-u^*_h\rangle_{\Gamma_h}\geq 0, & \forall u\in U_{ad}^h.
\end{eqnarray}
Moreover, the variational inequality is equivalent to the following projection equation
\begin{equation}\label{full-proj_equa}
u^*_h=\mathcal{P}_{[u_a,u_b]}\left(-\frac{1}{\alpha}\overline{R}_{h}\left( p^*_h|_{\Gamma_h} \right)\right),
\end{equation}
\end{theorem}
where $\overline{R}_h$ the piecewise constant interpolation operator from $L^2(\Gamma_h)$ onto $\overline{U}_h$.

$\mathbf{Remark4:}$ It is equivalent to $u^*_h=\mathcal{P}_{[u_a,u_b]}\left(-\frac{1}{\alpha} R_{h}^\circ\left( p^*_h|_{\Gamma} \right)\right),$
where $R_{h}^\circ$ denotes the projection from continuous function space on $\Gamma$ onto piecewise constant space on $\Gamma_{h}$, the details will be defined later.

\section{A Priori Error Estimates}\label{error}
In this section, we will obtain the error estimates between problems $(\mathbf{P})$ and $(\mathbf{P_h})$.
The mapping $G_h$, which is a homeomorphism between $\Omega^{\pm}_h$ and $\Omega^\pm$, is shown in \cite{Deckelnick2009Finite,gong2011mixed}, where $\Omega^{\pm}_h$ denote the uniform triangulation segmentation of $\Omega$ with interface $\Gamma_h$. Abbreviating $g_h:=G_{h|\Gamma_h}$ as a pull-back operator from $\Gamma_h$ to $\Gamma$.
We denote by $R_hv:=\left[\overline{R}_h(v\circ g_h)\right]\circ g_h^{-1}$ the piecewise constant interpolation operator onto $U_h$. Let $M_{\Gamma^i}$ be the midpoint of the element $\Gamma^i\in \Gamma$ and $M_{\Gamma_h^i}=g_h^{-1}(M_{\Gamma^i})$ be the point of the element $\Gamma_h^i\in \Gamma_h$.
\begin{equation}\label{0-inter}
\begin{aligned}
&R_h:C(\Gamma)\rightarrow {U}_h, \hspace{5mm} R_h u|_{\Gamma^i}=u(M_{\Gamma^i}),\\
&R_h^\circ:C(\Gamma)\rightarrow \overline{U}_h, \hspace{5mm} R_h^\circ u|_{\Gamma_h^i}=(R_{h}u)\circ g_h(M_{\Gamma_h^i}),
\end{aligned}
\end{equation}
where $R_h^\circ u=(R_hu)\circ g_h$ and furthermore $(R_hu)\circ g_h=\overline{R}_h(u\circ g_h)$.
It is well known \cite{ciarlet1991basic, apel2015finite} that for all $v\in H^2(\Gamma^i)\bigcap W^{1,\infty}(\Gamma^i)$, we have
\begin{equation}\label{interh}
\begin{aligned}
\left|\int_{\Gamma^i}(v-R_hv)d\Gamma^i\right|
\leq Ch^2|\Gamma^i|^{1/2}\|v\|_{H^2(\Gamma^i)}, ~~~\forall ~\Gamma^i\in \Gamma,
\end{aligned}
\end{equation}
and
\begin{equation}\label{interh1}
\left|\int_{\Gamma^i}(v-R_hv)d\Gamma^i\right|\leq Ch|\Gamma^i|\|v\|_{W^{1,\infty}(\Gamma^i)}, ~~~\forall ~\Gamma^i\in \Gamma.
\end{equation}

The projection operator $P_h$ is defined by $P_hv:=\left[\overline{P}_h(v\circ g_h)\right]\circ g_h^{-1}$ the $L^2(\Gamma)$-projection from $L^2(\Gamma)$ onto $U_h$, where $\overline{P}_h$ the $L^2(\Gamma_h)$-projection from $L^2(\Gamma_h)$ onto $\overline{U}_h$. Namely for each $v\in L^2(\Gamma)$,
$$
\langle P_h v, v_h\tau_h\rangle=\langle v,v_h\tau_h\rangle, ~~\forall v_h\in U_h,~\mbox{where } \tau_h=detDG_h^{-1}|(DG_h)^{T}\circ G_h^{-1}{\bf n}|.
$$

It follows from Bramble-Hilbert Lemma\cite{brenner2008mathematical,gunzburger1991analysis} for all $v\in H^s(\Gamma)$ we obtain
\begin{equation}\label{inter2}
\|v-P_hv\|_{L^2(\Gamma)}\leq Ch^s\|v\|_{H^s(\Gamma)}, ~~s\in [0,1].
\end{equation}
We need the following standard Sobolev extension operators.

$\mathbf{Remark5:}$ There are two extension operators $E^+:H^2(\Omega^+)\rightarrow H^2(\Omega)\cap H_0^1(\Omega)$ and $E^-:H^2(\Omega^-)\rightarrow H^2(\Omega)\cap H_0^1(\Omega)$, such that $(E^\pm v)\mid_{\Omega^\pm}=v$
and $\|E^\pm v\|_{H^2(\Omega)}\leq C\|v\|_{H^2(\Omega^\pm)}$. The details can be found in lemma 3.3 of \cite{Wu2019AN}.

$\mathbf{Remark6:}$ In the following, define a new linear interpolation operator $I_h^*:$ $I_h^{*}v=(I_hv)\mid_{K}$,$K\in \mathcal{T}_h^{non}$; $I_h^{*}v=(I_h^+v^+, I_h^-v^-)\mid_{K}$, where $I_h^{\pm}v^{\pm}=\left(I_hE^{\pm}v^{\pm}\right)\mid_{K^{\pm}}$, $K\in \mathcal{T}_h^{int}$.

Since we are focus on the errors on the interface, we will need the following inverse estimates, see\cite{mateos2011saturation}.
\begin{lemma}\label{invers}
Let $v_h\in V_{h}=\{v\in C(\bar{\Omega})| v|_T\in P_1(T), ~\forall T\in \mathcal{T}_h \}$. Then, for every $2\leq q <\infty$, there exists $C>0$ such that
\begin{equation}\label{invers1}
\|v_h\|_{L^q(\partial \Omega)}\leq Ch^{-1/q}\|v_h\|_{L^q(\Omega)},
\end{equation}
and
\begin{equation}\label{invers2}
\|v_h\|_{H^1(\partial \Omega)}\leq Ch^{-3/2}\|v_h\|_{L^2(\Omega)}.
\end{equation}
\end{lemma}

We also need some boundedness results of the solution operators.
\begin{lemma}\label{bounded}
There exists $C>0$ independent of $h$ such that for every $z\in L^2(\Omega)$
\begin{align}
 \label{bounded2}
&\|A_h^*z\|_{L^\infty(\Gamma)}\leq C\|z\|_{L^2(\Omega)}, \\ \label{bounded3}
&\|A_h^*z\|_{L^{2}(\Gamma)}\leq C\|z\|_{L^2(\Omega)}, \\ \label{bounded4}
&\|A_h^*z\|_{L^2(\Omega)}\leq C\|z\|_{L^2(\Omega)}, \\ \label{bounded5}
&\|A_h^*z\|_{H^1(\Gamma)}\leq C\|z\|_{L^2(\Omega)}.
\end{align}

\end{lemma}
\begin{proof}
For the first one, it follows from the usual trace theory, the fact that $p>2$, usual Sobolev embeddings, Theorem(8.5.3) in \cite{brenner2008mathematical} and the continuous inclusion $L^2(\Omega)\subset W^{-1,p\prime}(\Omega)$, we obtain
$$
\|A_h^*z\|_{L^\infty(\Gamma)}\leq C\|A_h^*z\|_{W^{1-1/p,p}(\Gamma)}\leq C(\|A_h^*z\|_{W^{1-1/p,p}(\partial \Omega^+)}+\|A_h^*z\|_{W^{1-1/p,p}(\partial \Omega^-)})
$$
$$
\leq C(\|A_h^*z\|_{W^{1,p}(\Omega^+)}+\|A_h^*z\|_{W^{1,p}(\Omega^-)})\leq C\|A_h^*z\|_{W^{1,p}(\Omega)}\leq C\|A^*z\|_{W^{1,p}(\Omega)}\leq  C\|z\|_{L^2(\Omega)}.
$$

For the second statement, using the trace Theorem, we have
$$
\|A_h^*z\|_{L^{2}(\Gamma)}\leq \|A_h^*z\|_{L^{2}(\partial\Omega^+)}+\|A_h^*z\|_{L^{2}(\partial\Omega^-)}
\leq \|A_h^*z\|_{H^1(\Omega^+)}+\|A_h^*z\|_{H^1(\Omega^-)}
$$
$$
\leq 2\|A_h^*z\|_{H^1(\Omega)}\leq C \|z\|_{L^2(\Omega)}.
$$

For the third statement, the embedding $H^1(\Omega)\rightarrow L^2(\Omega)$ yields
$$
\|A_h^*z\|_{H^1(\Omega)}^2\leq Ca_h(A_h^*z, A_h^*z)=C(z, A_h^*z)\leq C\|z\|_{L^2(\Omega)}\|A_h^*z\|_{H^1(\Omega)},
$$
then we obtain this fourth estimate with the embedding $H^1(\Omega)\rightarrow L^2(\Omega)$.

For the last estimates, by using the linear interpolation operator $I_h^*$, the Lemma~\ref{invers}, the approximation properties of $I_h$, Theorem~\ref{theo_IFEM-adj}, the embedding theorem, the trace theorem and the corresponding boundedness of $A^*$, we have
$$
\|A_h^*z\|_{H^1(\Gamma)}\leq \|A_h^*z-I_h^*A^*z\|_{H^1(\Gamma)}
+\|I_h^*A^*z-A^*z\|_{H^1(\Gamma)}+\|A^*z\|_{H^1(\Gamma)}
$$
$$
\leq  (\|A_h^*z-I_h^*A^*z\|_{H^1(\partial\Omega^+)}
+\|A_h^*z-I_h^*A^*z\|_{H^1(\partial\Omega^-)})
$$
$$
+C\left(\sum\limits_{K \in \mathcal{T}_h^{int}}\left(\|I_h^+(A^*z)^+-(A^*z)^+\|^2_{H^1(\partial K^+)}+\|I_h^-(A^*z)^--(A^*z)^-\|^2_{H^1(\partial K^-)}\right)\right)^{1/2}
$$
$$
+(\|A^*z\|_{H^1(\partial\Omega^+)}+\|A^*z\|_{H^1(\partial\Omega^-)})
$$
$$
\leq Ch^{-3/2}(\|A_h^*z-I_h^*A^*z\|_{L^2(\Omega^+)}+\|A_h^*z-I_h^*A^*z\|_{L^2(\Omega^-)})
$$
$$
+C\left(h\sum\limits_{K \in \mathcal{T}_h^{int}}\left(\|E^+(A^*z)^+\|^2_{H^2( K)}+\|E^-(A^*z)^-\|^2_{H^2( K)}\right)\right)^{1/2}$$
$$+C(\|A^*z\|_{H^2(\Omega^+)}+\|A^*z\|_{H^2(\Omega^-)})
$$
$$
\leq Ch^{-3/2}(\|A_h^*z-A^*z\|_{L^2(\Omega)}+\|A^*z-I_hA^*z\|_{L^2(\Omega)})+Ch^{1/2}\|A^*z\|_{\widetilde{H}^2(\Omega)}+C\|A^*z\|_{\tilde{H}^2(\Omega)}
$$

$$
\leq Ch^{-3/2}(h^2\|z\|_{L^2(\Omega)}+h^2\|A^*z\|_{\widetilde{H}^2(\Omega)})
+Ch^{1/2}\|z\|_{L^{2}(\Omega)}+C\|z\|_{L^2(\Omega)}\leq C\|z\|_{L^2(\Omega)}.
$$
\end{proof}

We will make the following assumption:
\par
$\mathbf{Assumption}$ 1:\label{asum1}
Let $K_1=\cup\{\Gamma^i| u^*|_{\Gamma^i}\notin H^2(\Gamma^i)\}$ and $K_2=\Gamma-K_1$. We suppose that there exists a positive constant $C$ independent of $h$ such that measure($K_1$)$\leq Ch$.

We state the main theorem in the following:
\begin{theorem}\label{main-th}
Let $(u^*,y^*,p^*)$ and $(u_h^*,y_h^*,p_h^*)$ be the solutions of the problems $(\mathbf{P})$ and $(\mathbf{P_h})$ respectively. Then there exists a constant $C>0$, independent of $h$, such that
\begin{eqnarray}\nonumber
&\|u^*-\tilde{u}^*_h\|_{L^{2}(\Gamma)} \\ \label{theo1}
&\leq Ch^{3/2}\left(\|u^*\|_{H^{1/2}(\Gamma)}+\|y_d\|_{C^{0,\lambda}(\overline{\Omega})}+\|f\|_{L^2(\Omega)}+\|\tilde{y}_{u^*}\|_{H^3\left(N\left(\Gamma,\rho\right)\right)}\right),\\ \nonumber
&\|y^*-y^*_h\|_{L^2(\Omega)} \\ \label{theo2}
&\leq Ch^{3/2}\left(\|u^*\|_{H^{1/2}(\Gamma)}+\|y_d\|_{C^{0,\lambda}(\overline{\Omega})}+\|f\|_{L^2(\Omega)}+\|\tilde{y}_{u^*}\|_{H^3\left(N\left(\Gamma,\rho\right)\right)}\right),\\
\nonumber
&\|p^*-p^*_h\|_{L^{2}(\Gamma)}+\|p^*-p^*_h\|_{L^2(\Omega)} \\ \label{theo3}
&\leq Ch^{3/2}\left(\|u^*\|_{H^{1/2}(\Gamma)}+\|y_d\|_{C^{0,\lambda}(\overline{\Omega})}+\|f\|_{L^2(\Omega)}+\|\tilde{y}_{u^*}\|_{H^3\left(N\left(\Gamma,\rho\right)\right)}\right),
\end{eqnarray}
\end{theorem}
where $\tilde{u}^*_h$ is defined later in (\ref{postprocessing}). Before we prove this Theorem, we first give some Lemmas in the following.

\begin{lemma}\label{lm4-2}
Let assumption 1 hold, then we have the following estimate
\begin{equation}\label{ineq2}
\begin{aligned}
&\|A^*_h(S_hu^*-y_d)-A^*_h(S_hR_h u^*-y_d)\|_{L^{2}(\Gamma)}\\
&\quad\leq C h^{2}\left(\|u^*\|_{H^{1/2}(\Gamma)}+\|y_d\|_{C^{0,\lambda}(\overline{\Omega})}+\|f\|_{L^2(\Omega)}
\right).
\end{aligned}
\end{equation}
\end{lemma}

\begin{proof}
Note that for all $v_h\in H_h(\Omega)$, the coercivity of the bilinear form with adjoint state weak form $a_h(p_h,v_h )=(S_hu^*-S_hR_h u^*, v_h)$, and the Cauchy-Schwarz inequality, we take $v_h=A^*_h(S_hu^*-S_hR_h u^*)$ and note that $p_h=A_h^*(S_hu^*-S_hR_h u^*)$ to obtain
$$
\|A^*_h(S_hu^*-S_hR_h u^*)\|_{H^1(\Omega)}^2\leq Ca_h(A^*_h(S_hu^*-S_hR_h u^*), A^*_h(S_hu^*-S_hR_h u^*))
$$
$$
\leq C\|S_hu^*-S_hR_h u^*\|_{L^2(\Omega)}\|A^*_h(S_hu^*-S_hR_h u^*)\|_{H^1(\Omega)},
$$
thus we have$\|A^*_h(S_hu^*-S_hR_h u^*)\|_{H^1(\Omega)}\leq \|S_hu^*-S_hR_h u^*\|_{L^2(\Omega)}$. Since
$\|A^*_h(S_hu^*-y_d)-A^*_h(S_hR_h u^*-y_d)\|_{L^{2}(\Gamma)}
=\|A^*_h(S_hu^*-S_h R_h u^*)\|_{L^{2}(\Gamma)}.$ Applying trace theorem\cite{brenner2008mathematical} on subdomains to obtain
\begin{equation}\label{I3-eq2}
\|A^*_h(S_hu^*-S_hR_h u^*)\|_{L^{2}(\Gamma)}\leq C\|A^*_hS_h\left(u^*-R_hu^*\right)\|_{H^{1}(\Omega)}\leq C\|S_h(u^*-R_h u^*)\|_{L^2(\Omega)}.
\end{equation}

Noting that for all $u\in L^2(\Gamma)$, and $z\in L^2(\Omega)$, using the weak forms (\ref{semeq}) of state  and the weak forms (\ref{dis-adj}) of the corresponding adjoint state, it is shown that
\begin{equation}\label{I3-eq3}
(S_hu,z)=(f, A_h^*z)+\langle u, A_h^*z \rangle-\sum_{T\in\mathcal{T}_h^{int}}\sum_{s=\pm}\int_{T\cap\Omega^{s}}\beta\nabla y_{\sigma u}\nabla v_h dx+(y_{\sigma u},z).
\end{equation}

Equation (\ref{I3-eq3}) with $u=R_h u^*$, we have
\begin{equation}\label{new1}
(S_hR_hu^*,z)=(f, A_h^*z)+\langle R_hu^*, A_h^*z \rangle-\sum_{T\in\mathcal{T}_h^{int}}\sum_{s=\pm}\int_{T\cap\Omega^{s}}\beta\nabla y_{\sigma u}\nabla v_h dx+(y_{\sigma u},z).
\end{equation}
Subtracting the above equation by taking $u=u^*$ in equation (\ref{I3-eq3}), we obtain
\begin{equation}\label{I3-eq4}
\begin{aligned}
(S_h(u^*-R_h u^*),z)&=\langle u^*-R_hu^*, A_h^*z \rangle.
\end{aligned}
\end{equation}

Then, we estimate the right hand side of (\ref{I3-eq2}), with $z=S_h(u^*-R_h u^*)$, using (\ref{bounded2}) of Lemma~\ref{bounded}, we have
\begin{equation}\label{J-eq1}
\begin{aligned}
&\|S_h(u^*-R_h u^*)\|_{L^2(\Omega)}^2\\
&\quad=\langle u^*-R_hu^*, A_h^*S_h(u^*-R_h u^*)\rangle \\
&\quad\leq  \left|\int_\Gamma (u^*-R_h u^*)d\Gamma\right|\|A_h^*S_h(u^*-R_h u^*)\|_{L^{\infty}(\Gamma)}  \\
&\quad\leq C\|S_h(u^*-R_h u^*)\|_{L^{2}(\Omega)}\left|\int_\Gamma (u^*-R_h u^*)d\Gamma\right|.
\end{aligned}
\end{equation}
Using the Assumption 1, $\left(\ref{interh}\right)$ and $\left(\ref{interh1}\right)$, we have
\begin{equation}\label{J1}
\begin{aligned}
&\left|\int_\Gamma (u^*-R_h u^*)d\Gamma\right| \\
\quad&\leq\sum\limits_{\Gamma^i\in \Gamma, \Gamma^i\subset K_1}\left|\int_{\Gamma^i}(u^*-R_h u^*)d\Gamma^i\right|
+\sum\limits_{\Gamma^i\in \Gamma, \Gamma^i\subset K_2}\left|\int_{\Gamma^i}(u^*-R_h u^*)d\Gamma^i\right|\\
\quad&\leq C\left[\sum\limits_{\Gamma^i\in \Gamma, \Gamma^i\subset K_1}h|\Gamma^i|\|u^*\|_{W^{1,\infty}(\Gamma^i)}
+\sum\limits_{\Gamma^i\in \Gamma, \Gamma^i\subset K_2}h^2|\Gamma^i|^{1/2}\|u^*\|_{H^2(\Gamma^i)}\right]\\
\quad&\leq C\left[h|K_1|\|u^*\|_{W^{1,\infty}(K_1)}
+h^2|K_2|^{1/2}\|u^*\|_{H^2(K_2)}\right]\\
\quad&\leq C\left[h^2\|u^*\|_{W^{2,\infty}(\Gamma)}
+h^2\|u^*\|_{H^2(\Gamma)}\right]\\
\quad&\leq Ch^2\left[\|p^*\|_{W^{2,\infty}(\Gamma)}
+\|p^*\|_{H^2(\Gamma)}\right].
\end{aligned}
\end{equation}

Note that the embedding $W^{2,\infty}(\Gamma)\rightarrow H^2(\Gamma)$, a regularity results of more general situations in \cite{li2010analysis} (also see Lemma 2.4 in \cite{apel2015finite}), and the embedding $H^2(\Omega^s)\rightarrow C^{0,\lambda}(\overline{\Omega}^s)(s=\pm, 0<\lambda<1)$, we have
\begin{equation}\label{J123}
\begin{aligned}
&\|p^*\|_{H^2(\Gamma)}\leq C\|p^*\|_{W^{2,\infty}(\Gamma)}\leq C(\|p^*\|_{W^{2,\infty}(\partial \Omega^+)}+\|p^*\|_{W^{2,\infty}(\partial \Omega^-)})\\
&\leq C\left(\|y^*-y_d\|_{C^{0,\lambda}(\overline{\Omega}^+)}+\|y^*-y_d\|_{C^{0,\lambda}(\overline{\Omega}^-)}\right) \leq C\left( \|y^*\|_{\widetilde{H}^2(\Omega)}+\|y_d\|_{C^{0,\lambda}(\overline{\Omega})}\right).
\end{aligned}
\end{equation}
Thus, we have
\begin{equation}\label{J11}
\left|\int_\Gamma (u^*-R_h u^*)d\Gamma\right|\leq Ch^2 (\|y^*\|_{\widetilde{H}^2(\Omega)}+\|y_d\|_{C^{0,\lambda}(\overline{\Omega})}).
\end{equation}

From (\ref{J-eq1})--(\ref{J11}), we obtain
\begin{equation*}
\begin{aligned}
&\|A^*_h(S_hu^*-y_d)-A^*_h(S_hR_h u^*-y_d)\|_{L^{2}(\Gamma)}\\
&\quad\leq Ch^{2}\left(\|u^*\|_{H^{1/2}(\Gamma)}+\|y_d\|_{C^{0,\lambda}(\overline{\Omega})}+\|f\|_{L^2(\Omega)}
\right).
\end{aligned}
\end{equation*}
\end{proof}

\begin{lemma}\label{lm4-3}
Let Assumption 1 hold, then the following estimate
\begin{equation}\label{ineq3}
\begin{aligned}
&\|p^*-A^*_h(S_hR_h u^*-y_d)\|_{L^{2}(\Gamma)}\\
&\quad\leq Ch^{3/2}\left(\|u^*\|_{H^{1/2}(\Gamma)}+\|y_d\|_{C^{0,\lambda}(\overline{\Omega})}+\|f\|_{L^2(\Omega)}+\|\tilde{y}_{u^*}\|_{H^3\left(N\left(\Gamma,\rho\right)\right)}\right)
\end{aligned}
\end{equation}
is valid.
\end{lemma}

\begin{proof}
We are going to introduce intermediate function and use the triangle inequality to get
\begin{equation*}
\begin{aligned}
&\|p^*-A^*_h(S_hR_h u^*-y_d)\|_{L^{2}(\Gamma)}\\
&\quad\leq \|p^*-A^*_h(S_hu^*-y_d)\|_{L^{2}(\Gamma)}+\|A^*_h(S_hu^*-y_d)
-A^*_h(S_hR_hu^*-y_d)\|_{L^{2}(\Gamma)}\\
&\quad=I_1+I_2.
\end{aligned}
\end{equation*}

To treat $I_1$, from Theorem~\ref{theo_IFEM-adj}, we obtain
\begin{equation}\label{I1}
\begin{aligned}
I_1&=\|p^*-A^*_h(S_hu^*-y_d)\|_{L^{2}(\Gamma)}\\
&\leq Ch^{3/2}(\|y_d\|_{L^2(\Omega)}+\|f\|_{L^2(\Omega)}+\|u^*\|_{H^{1/2}(\Gamma)}+\|\tilde{y}_{u^*}\|_{H^3\left(N\left(\Gamma,\rho\right)\right)}).
\end{aligned}
\end{equation}

For $I_2$, from Lemma~\ref{lm4-2}, we have
\begin{eqnarray}\label{I3}
I_2\leq Ch^{2}\left(\|u^*\|_{H^{1/2}(\Gamma)}+\|y_d\|_{C^{0,\lambda}(\overline{\Omega})}+\|f\|_{L^2(\Omega)}\right).
\end{eqnarray}

Combining (\ref{I1}) and (\ref{I3}), and noting the embedding $C^{0,\lambda}(\overline{\Omega})\rightarrow L^2(\Omega)$, we have
\begin{equation}
\begin{aligned}
&\|p^*-A^*_h(S_hR_h u^*-y_d)\|_{L^{2}(\Gamma)}\\
&\quad\leq Ch^{3/2}\left(\|u^*\|_{H^{1/2}(\Gamma)}+\|y_d\|_{C^{0,\lambda}(\overline{\Omega})}+\|f\|_{L^2(\Omega)}+\|\tilde{y}_{u^*}\|_{H^3\left(N\left(\Gamma,\rho\right)\right)}\right)
\end{aligned}
\end{equation}
This ends the proof.
\end{proof}

\begin{lemma}{\label{lm4-1}}
The following error estimate
\begin{equation}\label{ineq1}
\begin{aligned}
&\|R_hu^*-\hat{u}_h^*\|_{L^{2}(\Gamma)}\\
&\quad\leq Ch^{3/2}\left(\|u^*\|_{H^{1/2}(\Gamma)}+\|y_d\|_{C^{0,\lambda}(\overline{\Omega})}+\|f\|_{L^2(\Omega)}+\|\tilde{y}_{u^*}\|_{H^3\left(N\left(\Gamma,\rho\right)\right)}\right)
\end{aligned}
\end{equation}
is valid, where $\hat{u}^*_h=u^*_h\circ g_h^{-1}\in U_h$.
\end{lemma}

\begin{proof}
Let us recall the continuous and discrete optimality conditions. On the one hand we can use the operator $R_h^\circ$ for the continuous optimality (\ref{opts3}). On the other hand, because the admissible sets for the continuous and discrete problem are different, we have to deal with the discrepancy between the continuous admissible set $U_{ad}$ and the discrete admissible set $U_{ad}^h$, thus we can test the discrete optimality (\ref{dopts3}) with $R_h^\circ u^*\in U_{ad}^h$. It is obvious that the continuous optimality condition (\ref{opts3}) also satisfies in pointwise form, i.e.,
$$
\langle \alpha u^*(x)+ p^*(x), u-u^*(x)\rangle\geq 0,  ~~\forall u\in [u_a, u_b] \mbox{ and for a.a.  } x\in \Gamma.
$$

We take $x$ as the midpoint in each $\Gamma_h^i$, and integral over $\Gamma_h^i$, and sum over all the elements, we obtain $\langle \alpha R_h^\circ u^*(x)+ R_h^\circ p^*(x), R_h^\circ u(x)-R_h^\circ u^*(x)\rangle_{\Gamma_h}\geq 0,  ~~\forall R_h^\circ u(x)\in U_{ad}^h.$ If we choose $R_h^\circ u=u^*_h$ in the above variational inequalities, we have $\langle \alpha R_h^\circ u^*+ R_h^\circ p^*, u^*_h-R_h^\circ u^*\rangle_{\Gamma_h}\geq 0,$ taking $u=R_h^\circ u^*\in U_{ad}^h$ in (\ref{dopts3}), we have $\langle \alpha u^*_h+ p^*_h, R_h^\circ u^*-u^*_h\rangle_{\Gamma_h}\geq 0,$ by adding these two variational inequalities, we obtain $\langle \alpha (R_h^\circ u^*-u^*_h)+ R_h^\circ p^*-p^*_h, u^*_h-R_h^\circ u^*\rangle_{\Gamma_h}\geq 0.$ This yields
\begin{equation}\label{u-1}
\alpha \|R_h u^*-\hat{u}^*_h\|_{L^{2}(\Gamma)}^2\leq C \|R_h^\circ u^*-u^*_h\|_{L^{2}(\Gamma_h)}^2\leq C\langle R_h^\circ p^*-p^*_h, u^*_h-R_h^\circ u^*\rangle_{\Gamma_h}.
\end{equation}
\indent We will insert appropriate intermediate functions into the right hand side of (\ref{u-1}) to get
\begin{equation}\label{u-2}
\begin{aligned}
&\alpha \|R_hu^*-\hat{u}^*_h\|_{L^{2}(\Gamma)}^2\\
&\quad\leq C \|R_h^\circ u^*-u^*_h\|_{L^{2}(\Gamma_h)}^2\\
&\quad\leq C[\langle R_h^\circ p^*-p^*, u^*_h-R_h^\circ u^*\rangle_{\Gamma_h}
+\langle p^*-A^*_h(S_hR_h u^*-y_d), u^*_h-R_h^\circ u^*\rangle_{\Gamma_h}\\
&\quad\quad+\langle A^*_h(S_hR_h u^*-y_d)-p_h^*, u^*_h-R_h^\circ u^*\rangle_{\Gamma_h}]\\
&\quad=C\left[H_1+H_2+H_3\right].
\end{aligned}
\end{equation}
\indent We will estimate $H_1$, $H_2$ and $H_3$ separately. Since $u^*_h-R_h^\circ u^*$ is constant on each element $\Gamma_h^i$, for $H_1$, we have
\vspace{-0.5cm}
\begin{equation}\label{H1}
\begin{aligned}
H_1&=\sum\limits_{\Gamma_h^i\in \Gamma_h}\int_{\Gamma_h^i}(R_h^\circ p^*-p^*)(u^*_h-R_h^\circ u^*)d\Gamma_h^i=\sum\limits_{\Gamma_h^i\in \Gamma_h}(u^*_h-R_h^\circ u^*)|_{\Gamma_h^i}\int_{\Gamma_h^i}(R_h^\circ p^*-p^*)d\Gamma_h^i\\
&\leq C\|u^*_h-R_h^\circ u^*\|_{L^{2}(\Gamma_h)}\left(\sum\limits_{\Gamma_h^i\in \Gamma_h}\frac{1}{|\Gamma_h^i|}\left(\int_{\Gamma_h^i}(R_h^\circ p^*-p^*)d\Gamma_h^i\right)^2\right)^{1/2}\\
&\leq C\|\hat{u}^*_h-R_hu^*\|_{L^{2}(\Gamma)}\left(\sum\limits_{\Gamma_h^i\in \Gamma_h}\frac{1}{|\Gamma_h^i|}\left(\int_{\Gamma_h^i}(R_h^\circ p^*-p^*)d\Gamma_h^i\right)^2\right)^{1/2}\\
&\leq C\left[\frac{\varepsilon}{2}\|\hat{u}^*_h-R_hu^*\|^2_{L^{2}(\Gamma)}+\frac{1}{2\varepsilon}\sum\limits_{\Gamma_h^i\in \Gamma_h}\frac{1}{|\Gamma_h^i|}\left(\int_{\Gamma_h^i}(R_h^\circ p^*-p^*)d\Gamma_h^i\right)^2\right]\\
&\leq C\left[\frac{\varepsilon}{2}\|\hat{u}^*_h-R_hu^*\|^2_{L^{2}(\Gamma)}+\frac{C}{2\varepsilon}h^4\|p^*\|_{H^{2}\left(\Gamma\right)}\right]\\
&\leq C\left[\frac{\varepsilon}{2}\|\hat{u}^*_h-R_hu^*\|^2_{L^{2}(\Gamma)}+\frac{C}{2\varepsilon}h^4\left(\|y^*\|_{L^{2}(\Omega)}+\|y_d\|_{C^{0,\lambda}(\overline{\Omega})}\right)\right],
\end{aligned}
\end{equation}
where we have used Young's inequality, $\left(\ref{interh}\right)$ and $\left(\ref{adj-regularity}\right)$.

For $H_2$, using the Cauchy-Schwarz inequality and Lemma~\ref{lm4-3}, we have
 \begin{equation}\label{H2}
\begin{aligned}
H_2&\leq C\|p^*-A^*_h(S_hR_h u^*-y_d)\|_{L^{2}(\Gamma)}\|\hat{u}^*_h-R_hu^*\|_{L^{2}(\Gamma)}\\
&\leq \frac{C}{2\varepsilon}\|p^*-A^*_h(S_hR_h u^*-y_d)\|_{L^{2}(\Gamma)}^2+\frac{\varepsilon}{2}\|\hat{u}^*_h-R_hu^*\|_{L^{2}(\Gamma)}^2\\
&\leq \frac{Ch^3}{2\varepsilon}\left(\|u^*\|^2_{H^{1/2}(\Gamma)}+\|y_d\|^2_{C^{0,\lambda}(\overline{\Omega})}+\|f\|^2_{L^2(\Omega)}+\|\tilde{y}_{u^*}\|^2_{H^3\left(N\left(\Gamma,\rho\right)\right)}\right)\\
&\quad+\frac{\varepsilon}{2}\|\hat{u}^*_h-R_hu^*\|_{L^{2}(\Gamma)}^2.
\end{aligned}
\end{equation}

For the last term $H_3$, applying the pull-back operator, using (\ref{I3-eq4}) with $z=S_h(\hat{u}^*_h-R_hu^*)$ and $u^*=\hat{u}^*_h$, we can obtain,
 \begin{equation}\label{H3}
\begin{aligned}
&H_3=\langle A^*_h (S_hR_h u^*-y_d)-A_h^*(S_h\hat{u}^*_h-y_d), (\hat{u}^*_h-R_hu^*)\rangle_{\tau_h} \\
&\leq C\langle A^*_h (S_hR_h u^*-S_h\hat{u}^*_h), \hat{u}^*_h-R_hu^*\rangle\\
&=-C\|S_h(\hat{u}^*_h-R_hu^*)\|^2_{L^2(\Omega)}.
\end{aligned}
\end{equation}

From (\ref{H1}), (\ref{H2}) and (\ref{H3}), and taking appropriate $\varepsilon$, we deduce the desired results.
\end{proof}

Next, we will give the proof of main Theorem~\ref{main-th}.
\begin{proof}
Using appropriate intermediate functions and applying the triangle inequality, we have
$$
\|y^*-y^*_h\|_{L^2(\Omega)}=\|Su^*-S_hu^*+S_hu^*-S_hR_h u^*+S_hR_h u^*-S_h\hat{u}^*_h\|_{L^2(\Omega)}
$$
$$
\leq \|Su^*-S_hu^*\|_{L^2(\Omega)}+ \|S_hu^*-S_hR_h u^*\|_{L^2(\Omega)}+ \|S_hR_h u^*-S_h\hat{u}^*_h\|_{L^2(\Omega)}
$$
$$
\hspace{-9cm}=I_1+I_2+I_3.
$$

For the first term $I_1$, it is associate with the IEFM error estimates for the state equation, thus it is bounded by Theorem~\ref{theo_IFEM}. For the second term $I_2$, it is bounded by (\ref{J-eq1}). For the third one $I_3$, we estimate it by the technique in the proof of (\ref{I3-eq4}), Lemma \ref{bounded} and Lemma \ref{lm4-1}. It is shown that
\begin{equation}\label{y}
\begin{aligned}
&\|y^*-y^*_h\|_{L^2(\Omega)}\\
&\quad\leq  Ch^{3/2}\left(\|u^*\|_{H^{1/2}(\Gamma)}+\|y_d\|_{C^{0,\lambda}(\overline{\Omega})}+\|f\|_{L^2(\Omega)}+ \|\tilde{y}_{u^*}\|_{H^3\left(N\left(\Gamma,\rho\right)\right)}\right).
\end{aligned}
\end{equation}

The error of the adjoint state on the interface and in the domain can be bounded by
\begin{equation}\label{y1}
\begin{aligned}
&\|p^*-p^*_h\|_{L^2(\Omega)}+\|p^*-p^*_h\|_{L^{2}(\Gamma)}  \\
&\quad\leq\|p^*-p^*_h(u)\|_{L^2(\Omega)}+\|A_h^*\left(y^*_h(u)-y_d\right)-A_h^*\left(y^*_h-y_d\right)\|_{L^2(\Omega)}\\
&\quad\quad+\|p^*-A_h^*(S_hR_h u^*-y_d)\|_{L^{2}(\Gamma)}+\|A_h^*(S_hR_h u^*-y_d)-A_h^*(S_h\hat{u}_h ^*-y_d)\|_{L^{2}(\Gamma)}\\
&\quad\leq\|p^*-p^*_h(u)\|_{L^2(\Omega)}+\|y^*_h(u)-y^*\|_{L^2(\Omega)}+\|y^*-y^*_h\|_{L^2(\Omega)}\\
&\quad\quad+\|p^*-A_h^*(S_hR_h u^*-y_d)\|_{L^{2}(\Gamma)}+\|S_hR_h u^*-S_h\hat{u}_h ^*\|_{L^{2}(\Omega)},
\end{aligned}
\end{equation}
where we used Lemma \ref{bounded}. Applying Theorem~\ref{theo_IFEM-adj}, Theorem~\ref{theo_IFEM}, Lemma~\ref{lm4-3}, Lemma~\ref{lm4-1} and noting
(\ref{y}), we have
\begin{eqnarray}\label{p}
\|p^*-p^*_h\|_{L^{2}(\Gamma)}+\|p^*-p^*_h\|_{L^2(\Omega)} \nonumber\\ \leq Ch^{3/2}\left(\|u^*\|_{H^{1/2}(\Gamma)}+\|y_d\|_{C^{0,\lambda}(\overline{\Omega})}+\|f\|_{L^2(\Omega)}+\|\tilde{y}_{u^*}\|_{H^3\left(N\left(\Gamma,\rho\right)\right)}\right).
\end{eqnarray}

For control $u^*$, its discrete counterparts as introduced in Section 3 is $u^*_h$, however, we define the projection $\tilde{u}^*_h$ of $p^*_h$ by
\begin{equation}\label{postprocessing}
\tilde{u}^*_h=\mathcal{P}_{[u_a,u_b]}\left(-\frac{1}{\alpha}p^*_h|_{\Gamma}\right).
\end{equation}

We note that the function $\tilde{u}^*_h$ is piecewise linear and continuous, and it does not belong to the space $U_{ad}^h$. In the following, we will show the approximation between $\tilde{u}^*_h$ and $u^*$. Using the projection operator, and together with (\ref{p}),  we have
\begin{equation}\label{p1}
\begin{aligned}
\alpha \|u^*-\tilde{u}^*_h\|_{L^{2}(\Gamma)}&=\alpha\left\|\mathcal{P}_{[u_a,u_b]}\left(-\frac{1}{\alpha}p^*\right)
-\mathcal{P}_{[u_a,u_b]}\left(-\frac{1}{\alpha}p^*_h|_{\Gamma}\right)\right\|_{L^{2}(\Gamma)} \\
&\leq C \|p^*-p^*_h\|_{L^{2}(\Gamma)}\\
&\leq Ch^{3/2}\left(\|u^*\|_{H^{1/2}(\Gamma)}+\|y_d\|_{C^{0,\lambda}(\overline{\Omega})}+\|f\|_{L^2(\Omega)}+\|\tilde{y}_{u^*}\|_{H^3\left(N\left(\Gamma,\rho\right)\right)}\right).
\end{aligned}
\end{equation}
\end{proof}

If we use the variational discretization concept first proposed in \cite{hinze2005variational} for distributed control problem, i.e. $U_{ad}^h=U_{ad}$, we can prove the following error estimates in a similar way.
\begin{theorem}\label{main-th1}
Let $(u^*,y^*,p^*)$ and $(u_h^*,y_h^*,p_h^*)$ be the solutions of the problems ($\mathbf{P}$) and ($\mathbf{P_h}$) with $U_{ad}^h=U_{ad}$ respectively. Then there exists a constant $C>0$, independent of $h$, such that
\begin{equation}\label{theo4}
\|u^*-u^*_h\|_{L^{2}(\Gamma)}\leq Ch^{3/2}\left(\|u^*\|_{H^{1/2}(\Gamma)} +\|y_d\|_{L^{2}(\Omega)}+\|f\|_{L^2(\Omega)}+\|\tilde{y}_{u^*}\|_{H^3\left(N\left(\Gamma,\rho\right)\right)}\right).
\end{equation}
\end{theorem}

\begin{proof}
It suffices to set $u=u_h^*\in U_{ad}$ in the continuous optimality (\ref{opts3}) and $u=u^*\in U_{ad}$ in discrete optimality (\ref{dopts3}) and add the corresponding inequalities. This directly gives
$$
 \alpha\|u^*-u^*_h\|_{L^{2}(\Gamma)}^2=\alpha \langle u^*-u^*_h, u^*-u^*_h\rangle=\alpha \langle u^*, u^*-u^*_h\rangle-\alpha \langle u^*_h, u^*-u^*_h\rangle
$$
$$
\leq\langle p^*_h-p^*, u^*-u^*_h\rangle \leq \|p^*-p^*_h\|_{L^{2}(\Gamma)}\|u^*-u^*_h\|_{L^{2}(\Gamma)}.
$$

Using Theorem~\ref{theo_IFEM-adj}, we have
$$
\|u^*-u^*_h\|_{L^{2}(\Gamma)}\leq Ch^{3/2} \left(\|u^*\|_{H^{1/2}(\Gamma)}+\|y_d\|_{L^2(\Omega)}+\|f\|_{L^2(\Omega)}+\|\tilde{y}_{u^*}\|_{H^3\left(N\left(\Gamma,\rho\right)\right)}\right).
$$
\end{proof}

\section{Numerical experiments}

In this section, we give an example that includes both control without constraints and with constraints to validate our theoretical analysis, it is a waterdrop-shaped interface $\Gamma(x_{1},x_{2})=\frac{9}{4}\left(x_{1}^{2}+x_{2}^{2}\right)^{2}-2x_{1}\left(x_{1}^{2}+x_{2}^{2}\right)+3x_{2}^{2}$. We set the calculation domain is $\Omega=[-1,1] \times[-1,1]$, the coefficient $\beta^{+}=10, \beta^{-}=1$, and the regularity parameter $\alpha=1$. The iteration algorithm used is the fixed-point iteration.

The algorithm is as follows:
\vspace{-3mm}
\begin{algorithm}
	\caption{ Algorithm for the solution of optimal control problem.}
	\label{alg:Framwork}
	\begin{algorithmic}[1] 
		\label{ code:fram:extract }
		\STATE Provide an initial $u_{h0}$ of the control function $u_h$;
		\label{code:fram:trainbase}
		\STATE Solve the equation of state in $y_h$ using the above $u_{h0}$;
		\label{code:fram:add}
		\STATE Solve the adjoint equation for $p_h$, being known $y_h$ and $y_d$;
		\label{code:fram:1classify}
		\STATE Compute the new control function $u_{h1}$ using the above $p_h$;
		\label{code:fram:select}
		\STATE If $|u_{h0}-u_{h1}|\leq1.0\times10^{-15}$, setting $u_h=u_{h1}$, else setting $u_{h0}=u_{h1}$ and goto step 2; 
		\label{code:fram:classify}
		\STATE Take the last computed control function $u_h$ to compute the $y_h$ and $p_h$.
	\end{algorithmic}
\end{algorithm}

\vspace{-3mm}
According to \cite{pinnau2008optimization}, this algorithm is convergent if the parameter $\alpha$ is large enough.

For the error functional, the experimental order of convergence is defined by
\begin{align*}
	\text{Order}=\frac{\log E(h_1)-\log E(h_2)}{\log h_1-\log h_2}.
\end{align*}


{\bf Case 1} The exact solution $y^{*}$ is constructed with a nonhomogeneous boundary condition and the control is considered without constraint. The optimal triple $\left(y^{*}, p^{*}, u^{*}\right)$ is given by
\begin{equation*}
	\begin{aligned}
		&y^{*}\left(x_{1}, x_{2}\right)= \begin{cases}&\hspace{-3mm}\frac{9}{4}\left(x_{1}^{2}+x_{2}^{2}\right)^{2}+3x_{2}^{2}, ~ if \left(x_{1}, x_{2}\right) \in \Omega^{+}, \\
			&\hspace{-3mm}2x_{1}\left(x_{1}^{2}+x_{2}^{2}\right), ~ if \left(x_{1}, x_{2}\right) \in \Omega^{-} .\end{cases} \\
		&p^{*}\left(x_{1}, x_{2}\right)= \begin{cases}&\hspace{-3mm}\beta^{-}\left(\frac{9}{4}\left(x_{1}^{2}+x_{2}^{2}\right)^{2}-2x_{1}\left(x_{1}^{2}+x_{2}^{2}\right)+3x_{2}^{2}\right) \left(x_{1}^{2}-1\right) \left(x_{2}^{2}-1\right),\\
 & \hspace{6cm} if \left(x_{1}, x_{2}\right) \in \Omega^{+}, \\
			&\hspace{-3mm}\beta^{+}\left(\frac{9}{4}\left(x_{1}^{2}+x_{2}^{2}\right)^{2}-2x_{1}\left(x_{1}^{2}+x_{2}^{2}\right)+3x_{2}^{2}\right) \left(x_{1}^{2}-1\right) \left(x_{2}^{2}-1\right), \\
  & \hspace{6cm}if \left(x_{1}, x_{2}\right) \in \Omega^{-} .\end{cases}
	\end{aligned}
\end{equation*}
\begin{equation*}
	u^{*}\left(x_{1}, x_{2}\right)=0, ~~for~\left(x_{1}, x_{2}\right) \in \Gamma.
\end{equation*}

From this, the functions $f, g, y_{d}$ can be determined by optimality system. An illustration of the geometry of the interface shape is depicted in Figure $\ref{interfaceshuidi}$.
\begin{figure}[ht]
	\centering
	\includegraphics[width=0.33\linewidth]{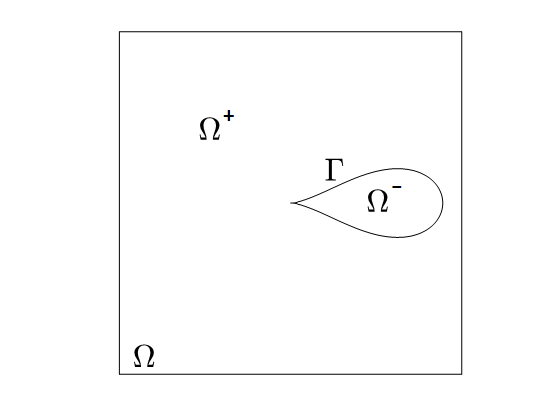}
	\caption{~The geometry of interface shape.}
	\label{interfaceshuidi}
\end{figure}

\begin{table}[ht]
	\centering
	\footnotesize
	\begin{center}
		\caption
		{\footnotesize The $L^{2}$ error and convergence order of the state $y^{*}$, the adjoint state $p^{*}$ and the control $u^{*}$ (without constraint).}
		\begin{tabular}{ccccccc}
			\hline
			$N$  &$||y^{*}-y^{*}_h||_{L^{2}(\Omega)}$  &$Order$  &$||p^{*}-p^{*}_h||_{L^{2}(\Omega)}$  &$Order$  &$||u^{*}-u^{*}_h||_{L^{2}(\Gamma)}$  &$Order$\\
			\hline
			32  &1.9091E-02  &\  &7.0139E-03  &\  &7.5853E-03  \ \\
			
			64  &4.6808E-03  &2.0280  &1.6853E-03  &2.0572  &2.3373E-03  &1.6984 \\
			
			128 &9.0432E-04  &2.3719  &3.7958E-04  &2.1505  &5.3963E-04  &2.1148\\
			
			256 &2.3197E-04  &1.9629  &9.3614E-05  &2.0196  &1.3689E-04  &1.9790\\
			
			512 &6.5059E-05  &1.8341  &2.2682E-05  &2.0452  &3.3092E-05  &2.0485\\
			\hline
		\end{tabular}
	\end{center}
\end{table}

\begin{table}[ht]
	\centering
	\footnotesize
	\begin{center}
		\caption
		{\footnotesize The $L^{\infty}$ error and convergence order of the state $y^{*}$, the adjoint state $p^{*}$ and the control $u^{*}$ (without constraint).}
		\begin{tabular}{ccccccc}
			\hline
			$N$  &$||y^{*}-y^{*}_h||_{L^{\infty}(\Omega)}$  &$Order$  &$||p^{*}-p^{*}_h||_{L^{\infty}(\Omega)}$  &$Order$  &$||u^{*}-u^{*}_h||_{L^{\infty}(\Gamma)}$  &$Order$\\
			\hline
			32  &1.1029E-01  &\  &5.9435E-02  &\  &1.2408E-02  \ \\
			
			64  &2.5647E-02  &2.1044  &1.5335E-02  &1.9545  &4.1848E-03  &1.5680 \\
			
			128 &5.7656E-03  &2.1533  &4.1636E-03  &1.8809  &9.9563E-04  &2.0715\\
			
			256 &1.5186E-03  &1.9248  &1.1247E-03  &1.8883  &2.4515E-04  &2.0219\\
			
			512 &4.1024E-04  &1.8882  &2.9108E-04  &1.9501  &6.5111E-05  &1.9127\\
			\hline
		\end{tabular}
	\end{center}
\end{table}

Tables 1-2 show the $L^{2}$ error and convergence order, the $L^{\infty}$ error and convergence order of the state $y^{*}$, the adjoint state $p^{*}$ and the control $u^{*}$ without constraint, respectively. We can intuitively obtain that the convergence order of $y^{*}, p^{*}$, $u^{*}$ is second order, which is superior to our theoretical result.

We present some figures to better characterize the numerical result: The numerical solution images of the state, the adjoint state and the control with $N=128$ are displayed in Figures $\ref{13state}$-$\ref{13control}$; The error images of the state, the adjoint state and the control with $N=128$ are displayed in Figures $\ref{13error}$. We can find that the error at the interface is larger than elsewhere, which is in line with common sense.

\begin{figure}[H]
	\centering
	\subfigure[Computed state]{
		\label{Fig.a}
		\includegraphics[width=0.36\textwidth]{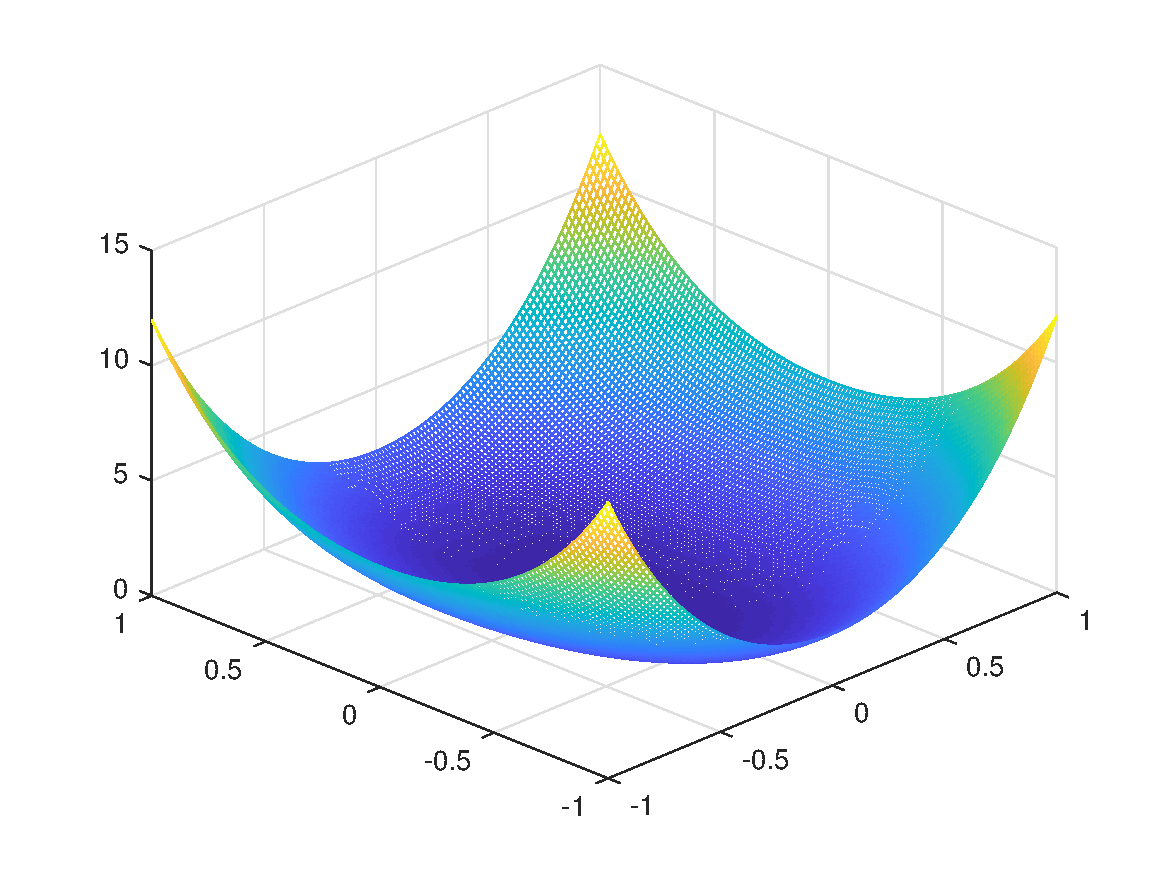}}
	\subfigure[Top view]{
		\label{Fig.b} \includegraphics[width=0.36\textwidth]{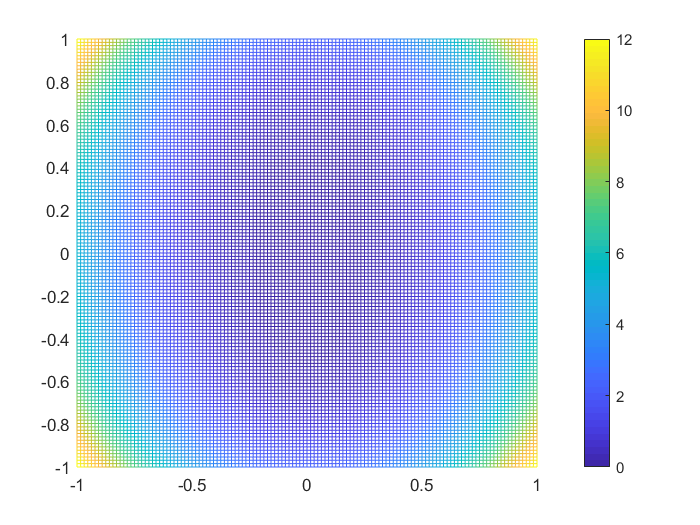}}
	\caption{~The computed state with N=128.}
	\label{13state}
\end{figure}

\begin{figure}[H]
	\centering
	\subfigure[Computed adjoint state]{
		\label{Fig.a1}
		\includegraphics[width=0.36\textwidth]{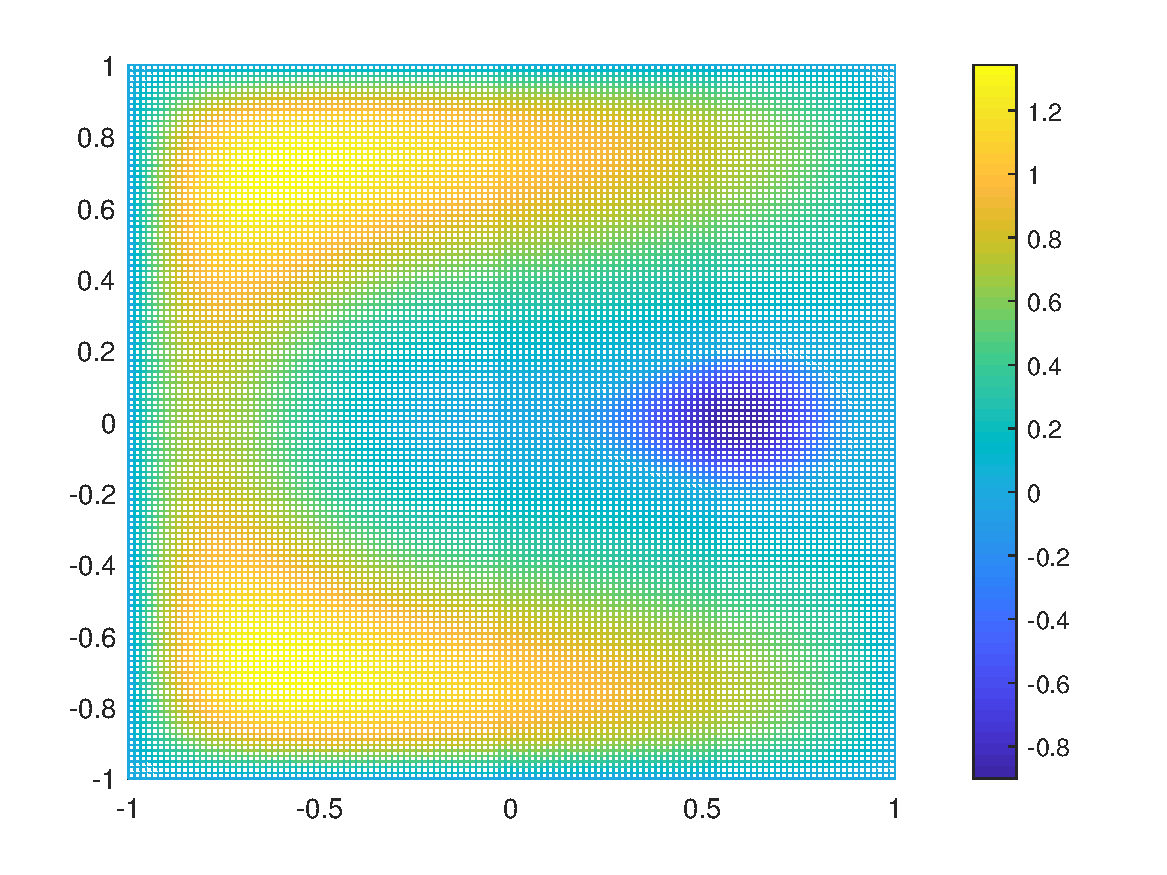}}
	\subfigure[Top view]{
		\label{Fig.b1} \includegraphics[width=0.36\textwidth]{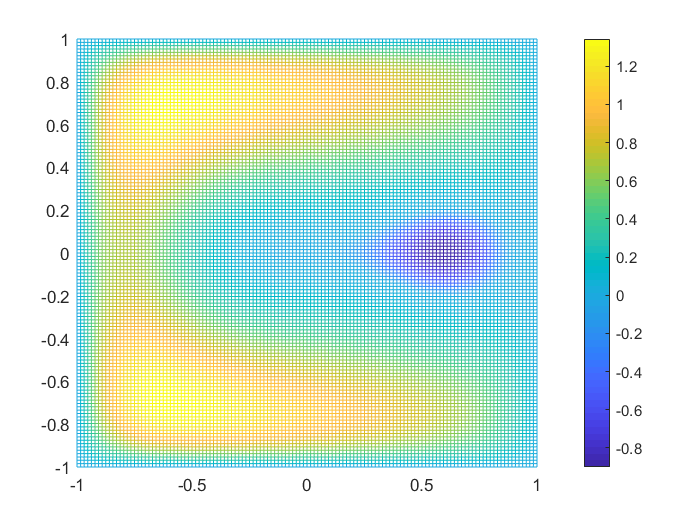}}
	\caption{~The computed adjoint state with N=128.}
	\label{13adjoint stste}
\end{figure}
\vspace{-0.3cm}
\begin{figure}[H]
	\centering
	\subfigure[Computed control]{
		\label{Fig.a2}	
		\includegraphics[width=0.36\textwidth]{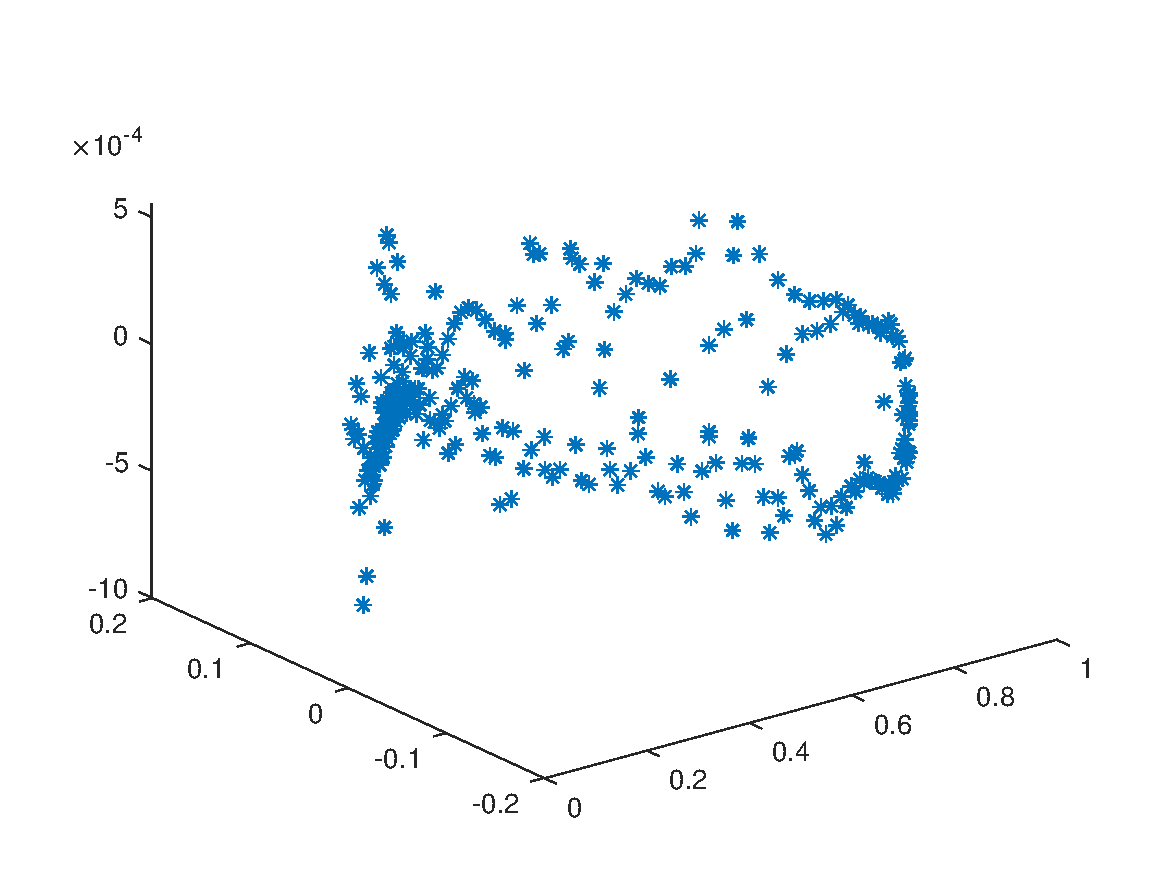}}
	\subfigure[Top view]{
		\label{Fig.b2}
		\includegraphics[width=0.36\textwidth]{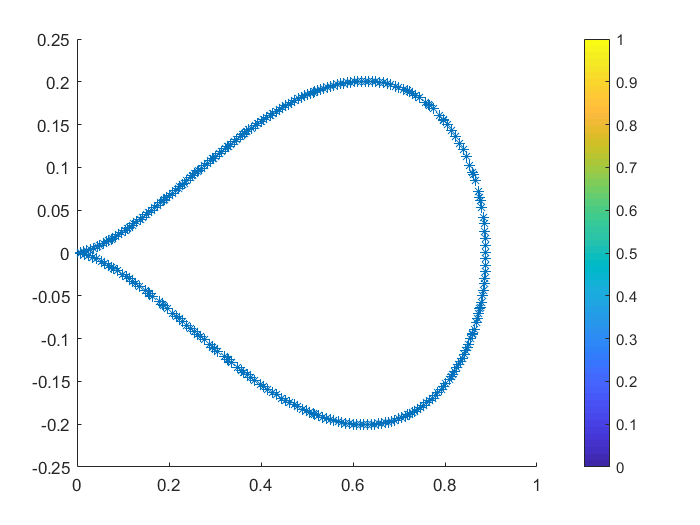}}
	\caption{~The computed control (without constraint) with N=128.}
	\label{13control}
\end{figure}
\vspace{-0.6cm}
\begin{figure}[H]
	\centering
	\subfigure[state error]{
		\label{Fig.a3}	
		\includegraphics[width=0.36\textwidth]{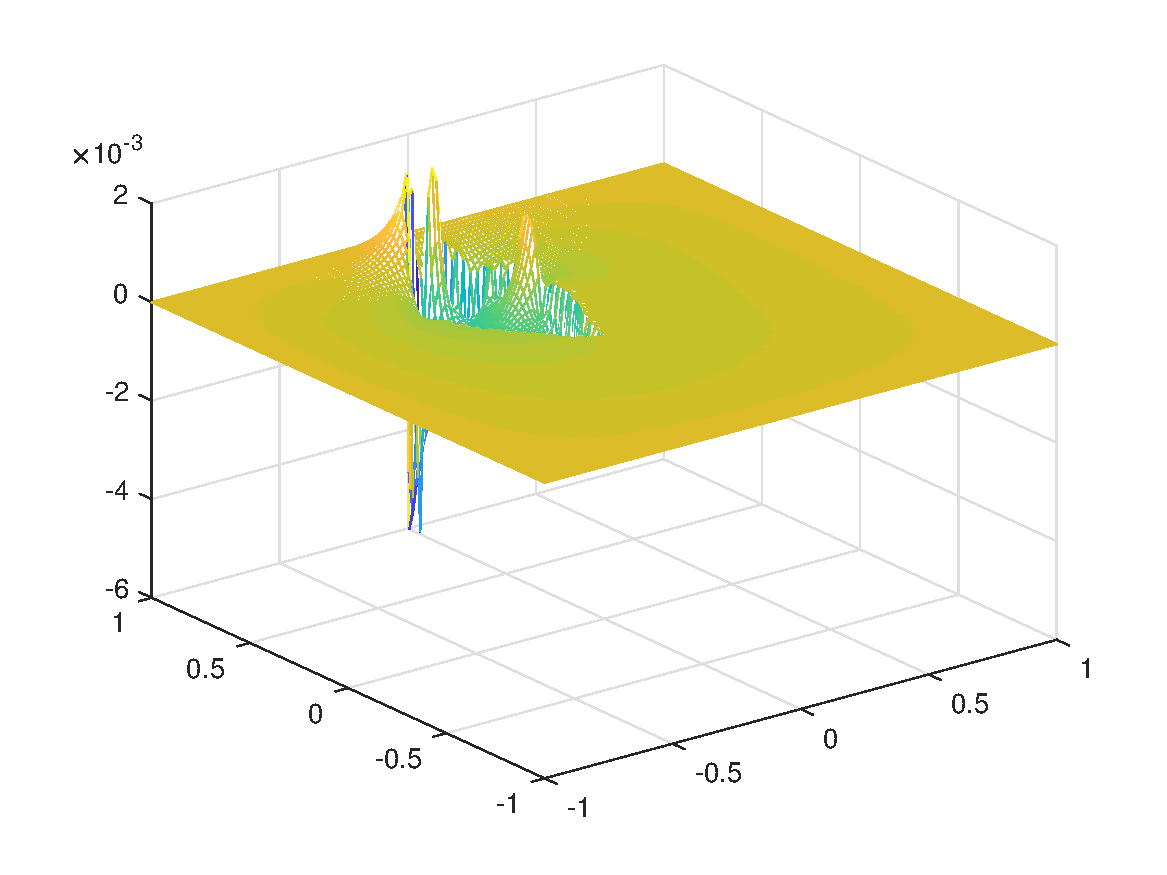}}
	\subfigure[adjoint state error]{
		\label{Fig.b3}
		\includegraphics[width=0.36\textwidth]{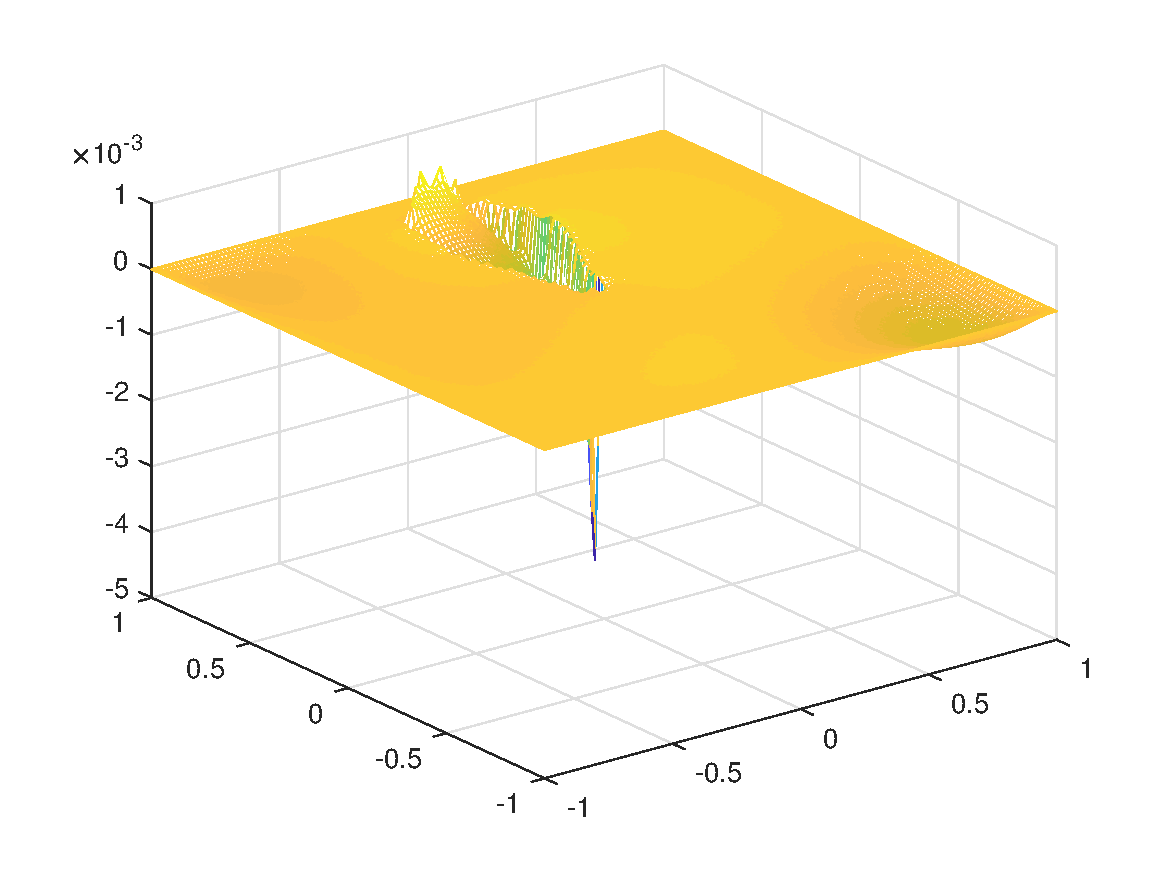}}
	\subfigure[control error]{
		\label{Fig.c}	
		\includegraphics[width=0.36\textwidth]{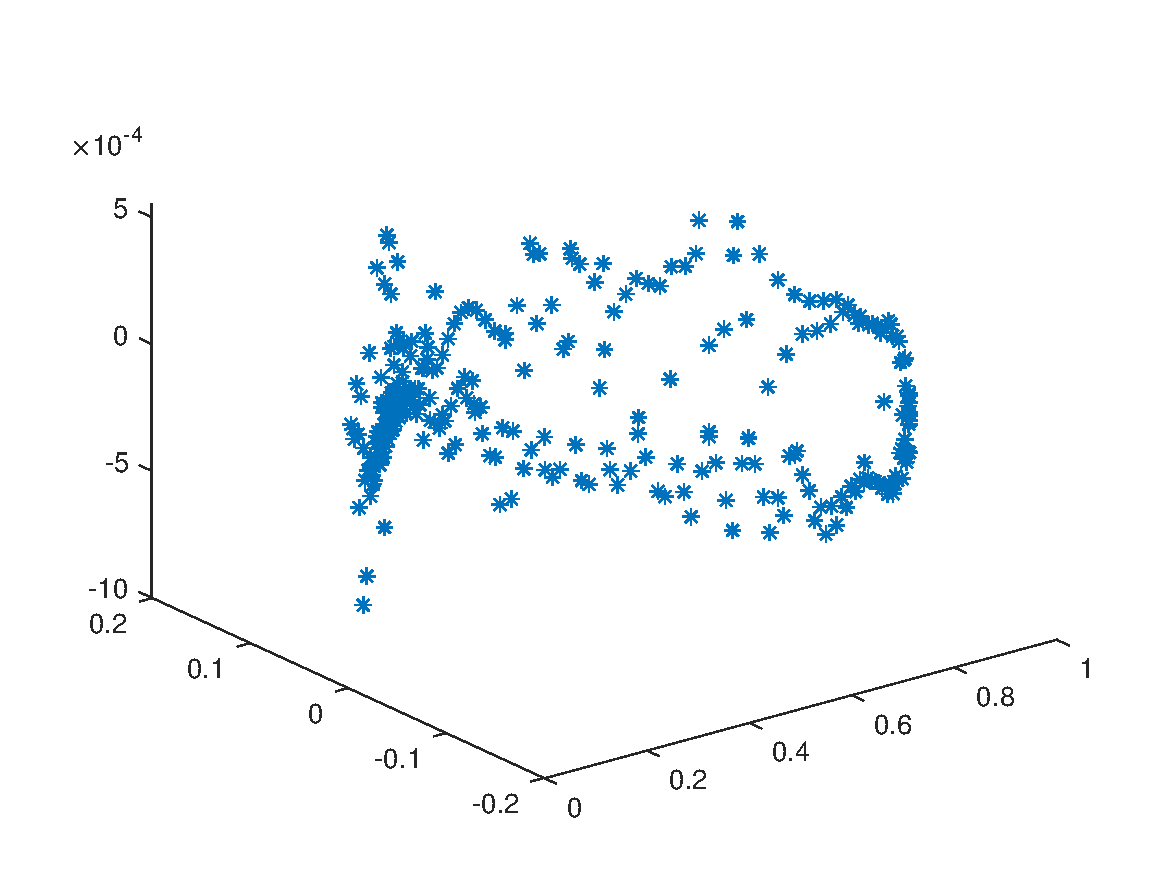}}
	\caption{ ~The error of $y^{*}, p^{*}, u^{*}$ (without constraint) with N=128.}
	\label{13error}
\end{figure}

{\bf Case 2} In this case, we consider the control variable with constraints as follows:
\begin{equation*}
	u^{*}\left(x_{1}, x_{2}\right)=\max\left(0,\sin(2\pi x_{1})\right), ~~for~\left(x_{1}, x_{2}\right) \in \Gamma.
\end{equation*}
Other quantities are set as in case 1.

\begin{table}[h]
	\centering
	\footnotesize
	\begin{center}
		\caption
		{\footnotesize The $L^{2}$ error and convergence order of the state $y^{*}$, the adjoint state $p^{*}$ and the control $u^{*}$ (with constraint).}
		\begin{tabular}{ccccccc}
			\hline
			$N$  &$||y^{*}-y^{*}_h||_{L^{2}(\Omega)}$  &$Order$  &$||p^{*}-p^{*}_h||_{L^{2}(\Omega)}$  &$Order$  &$||u^{*}-u^{*}_h||_{L^{2}(\Gamma)}$  &$Order$\\
			\hline
			32  &1.9201E-02  &\  &7.0154E-03  &\  &5.2374E-03  \ \\
			
			64  &4.6934E-03  &2.0304  &1.6855E-03  &2.0573  &1.7044E-03  &1.6196 \\
			
			128 &8.9902E-04  &2.3842  &3.7954E-04  &2.1508  &3.9090E-04  &2.1244\\
			
			256 &2.2221E-04  &2.0164  &9.3518E-05  &2.0210  &1.0352E-04  &1.9169\\
			
			512 &5.5037E-05  &2.0135  &2.2577E-05  &2.0504  &2.5156E-05  &2.0409\\
			\hline
		\end{tabular}
	\end{center}
\end{table}
\vspace{-0.5cm}
\begin{table}[h]
	\centering
	\footnotesize
	\begin{center}
		\caption
		{\footnotesize The $L^{\infty}$ error and convergence order of the state $y^{*}$, the adjoint state $p^{*}$ and the control $u^{*}$ (with constraint).}
		\begin{tabular}{ccccccc}
			\hline
			$N$  &$||y^{*}-y^{*}_h||_{L^{\infty}(\Omega)}$  &$Order$  &$||p^{*}-p^{*}_h||_{L^{\infty}(\Omega)}$  &$Order$  &$||u^{*}-u^{*}_h||_{L^{\infty}(\Gamma)}$  &$Order$\\
			\hline
			32  &1.1029E-01  &\  &5.9440E-02  &\  &1.2407E-02  \ \\
			
			64  &2.5641E-02  &2.1048  &1.5336E-02  &1.9546  &4.1846E-03  &1.5679 \\
			
			128 &5.7581E-03  &2.1548  &4.1636E-03  &1.8810  &9.9573E-04  &2.0713\\
			
			256 &1.5122E-03  &1.9290  &1.1245E-03  &1.8885  &2.4530E-04  &2.0212\\
			
			512 &4.0406E-04  &1.9040  &2.9088E-04  &1.9508  &6.5246E-05  &1.9106\\
			\hline
		\end{tabular}
	\end{center}
\end{table}

Similarly, tables 3-4 show the $L^{2}$ error and convergence order, the $L^{\infty}$ error and convergence order of the state $y^{*}$, the adjoint state $p^{*}$ and the control $u^{*}$ with constraint, respectively. It can be clearly seen that we get the error convergence order of all three variables is still second order when the control variable has constraints.

\begin{figure}[H]
	\centering
	\subfigure[Computed state]{
		\label{Fig.a4}
		\includegraphics[width=0.36\textwidth]{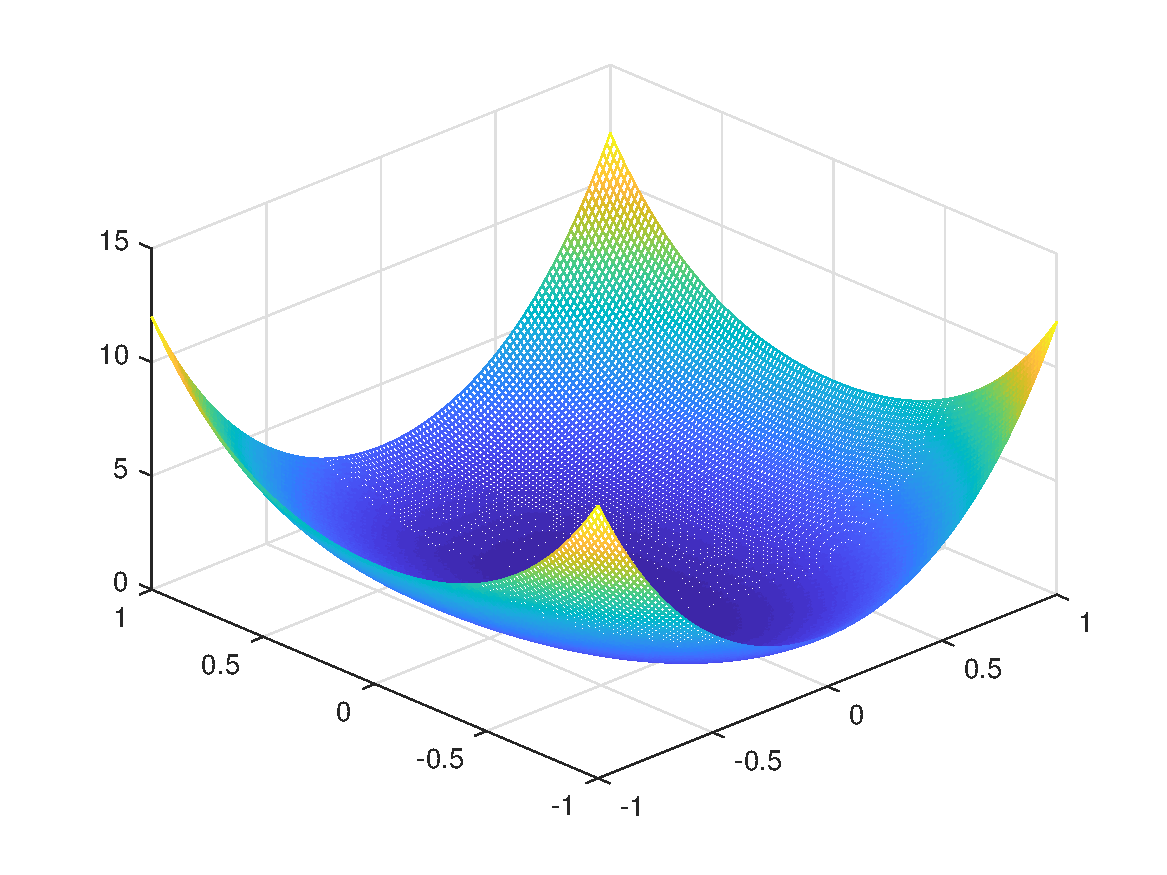}}
	\subfigure[Top view]{
		\label{Fig.b4}
		\includegraphics[width=0.36\textwidth]{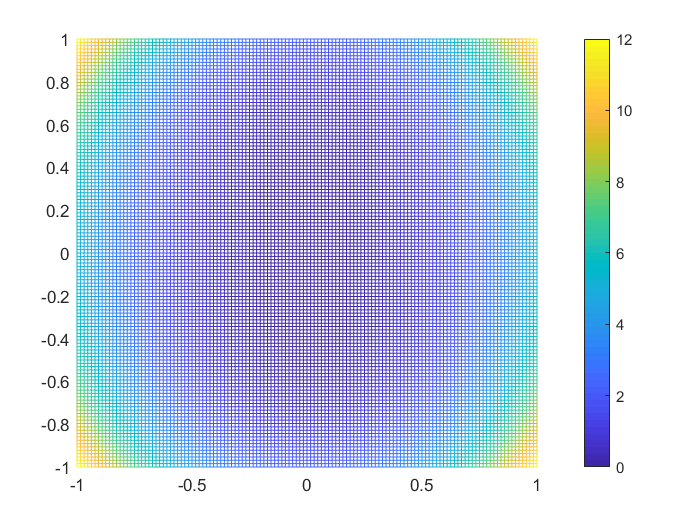}}
	\caption{~The computed state with N=128.}
	\label{13ysstate}
\end{figure}

\begin{figure}[H]
	\centering
	\subfigure[Computed adjoint state]{
		\label{Fig.a5}
		\includegraphics[width=0.36\textwidth]{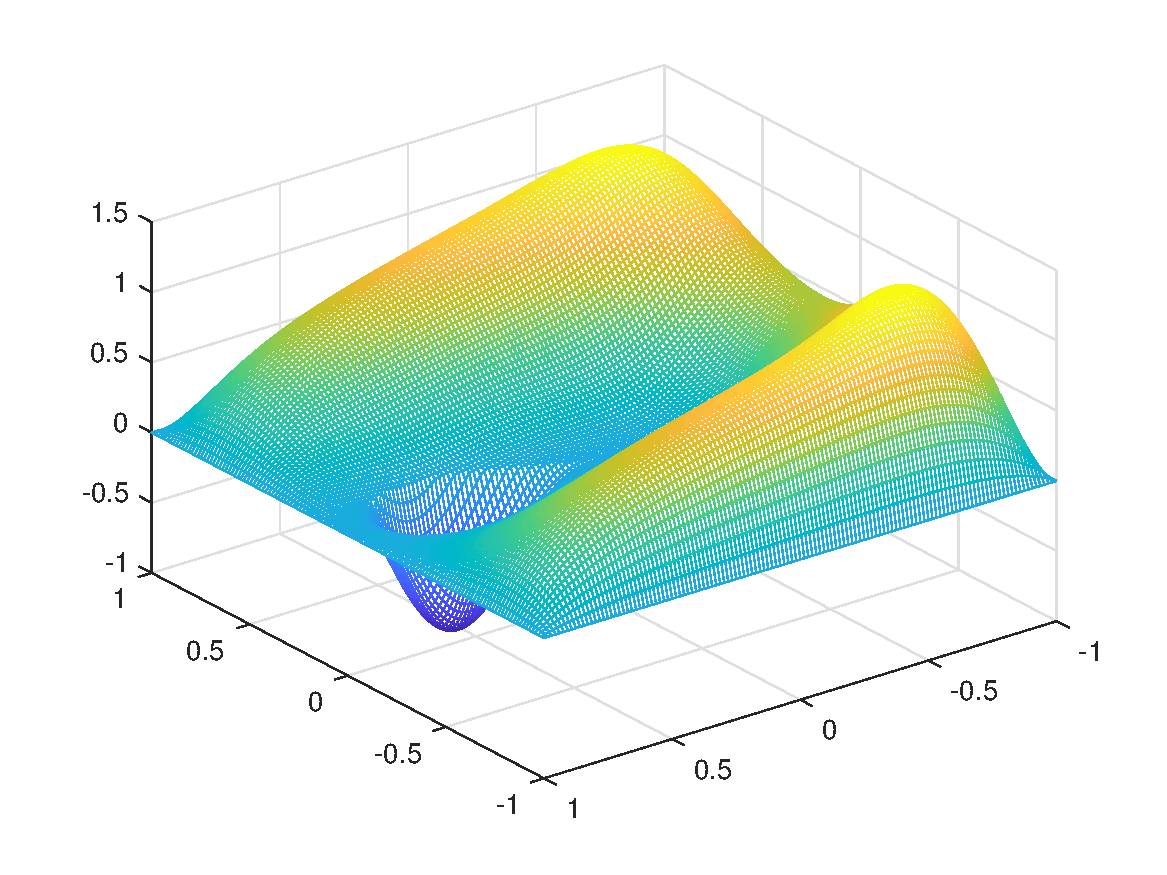}}
	\subfigure[Top view]{
		\label{Fig.b5}
		\includegraphics[width=0.36\textwidth]{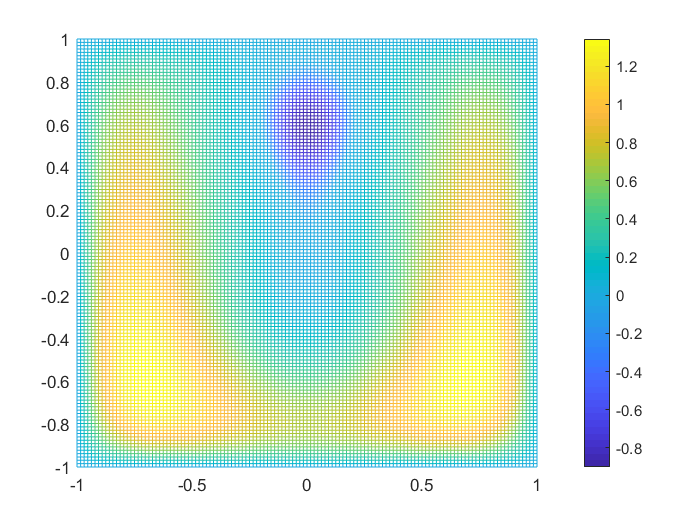}}
	\caption{~The computed adjoint state with N=128.}
	\label{13ysadjoint stste}
\end{figure}
\vspace{-0.3cm}
\begin{figure}[H]
	\centering
	\subfigure[Computed control]{
		\label{Fig.a6}	
		\includegraphics[width=0.36\textwidth]{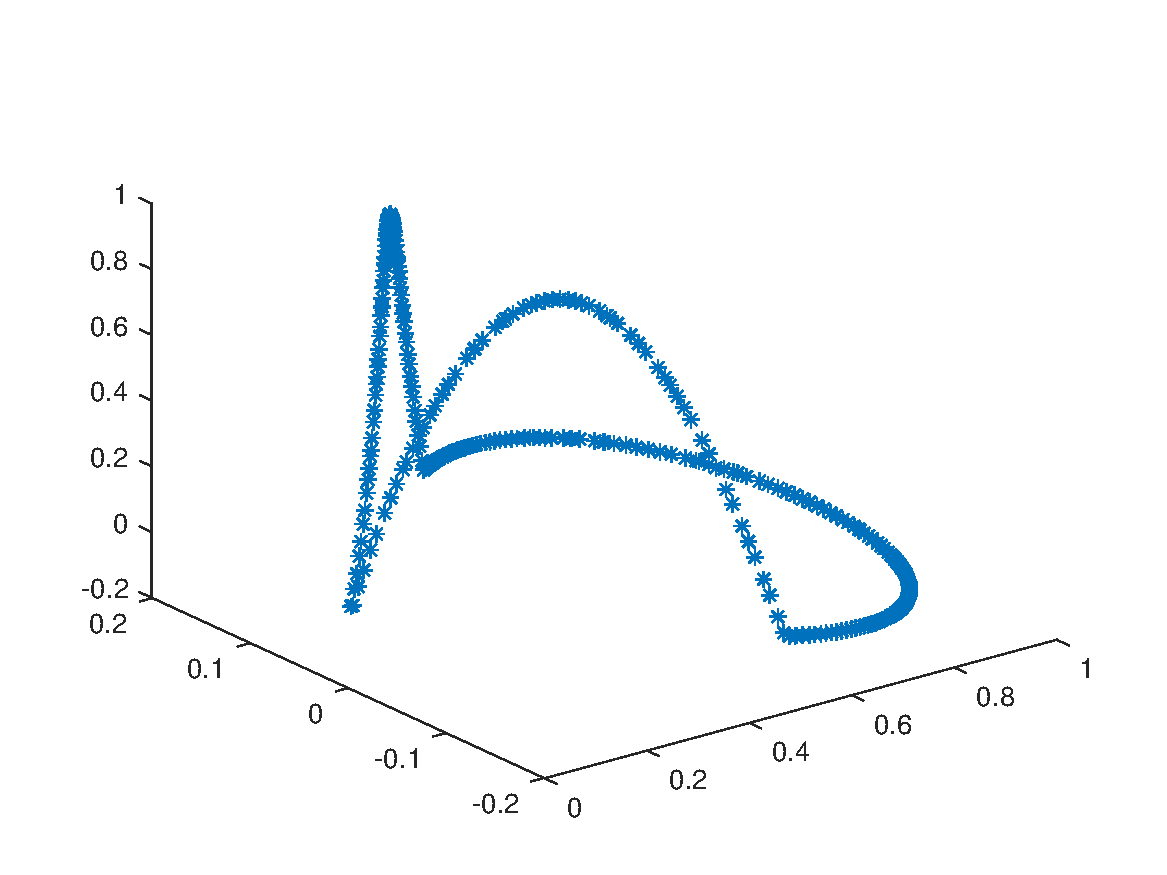}}
	\subfigure[Top view]{
		\label{Fig.b6}
		\includegraphics[width=0.36\textwidth]{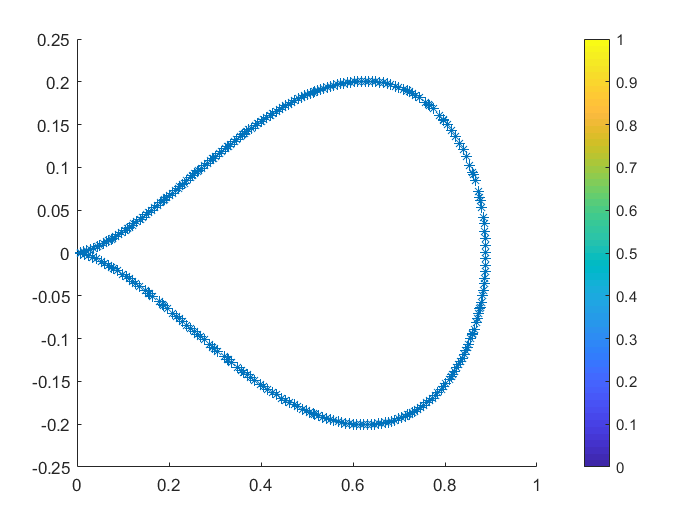}}
	\caption{~The computed control (without constraint) with N=128.}
	\label{13yscontrol}
\end{figure}
\vspace{-0.6cm}
\begin{figure}[H]
	\centering
	\subfigure[state error]{
		\label{Fig.a7}	
		\includegraphics[width=0.36\textwidth]{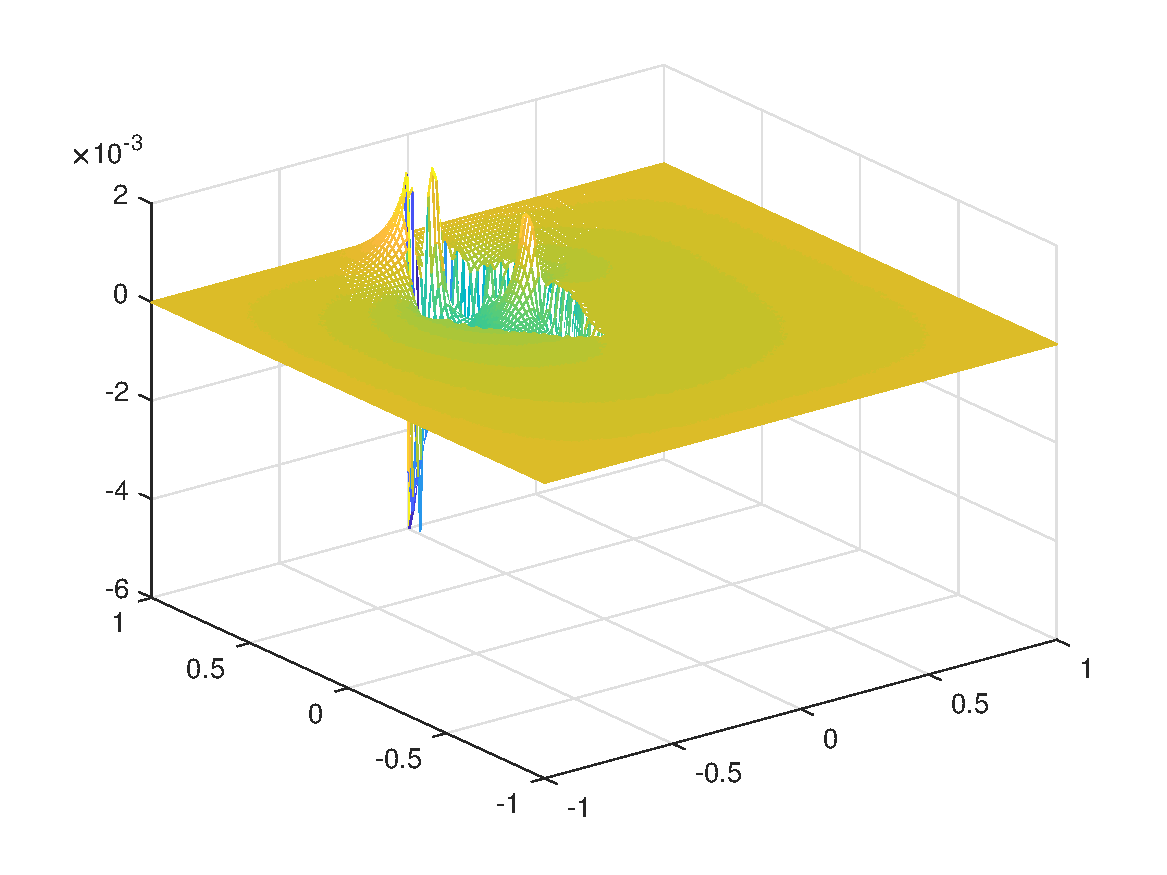}}
	\subfigure[adjoint state error]{
		\label{Fig.b7}
		\includegraphics[width=0.36\textwidth]{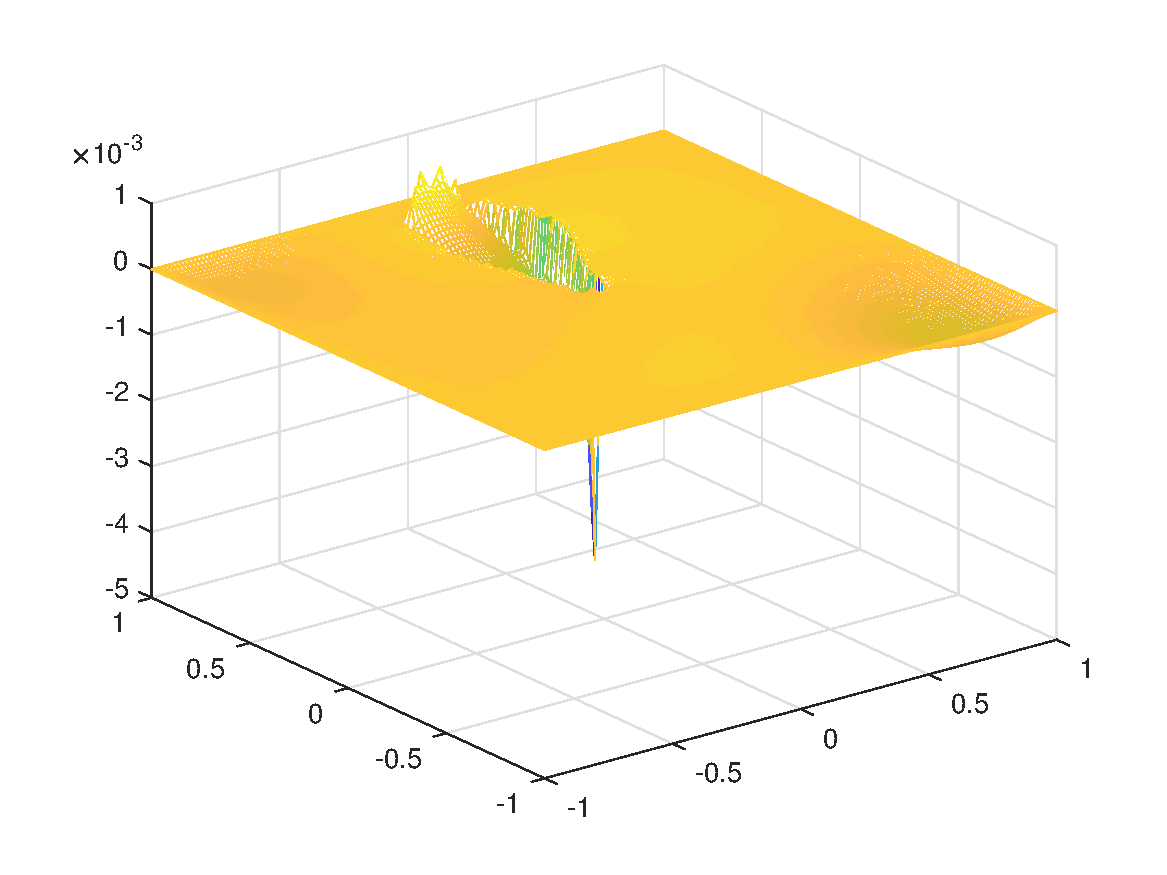}}
	\subfigure[control error]{
		\label{Fig.c1}	
		\includegraphics[width=0.36\textwidth]{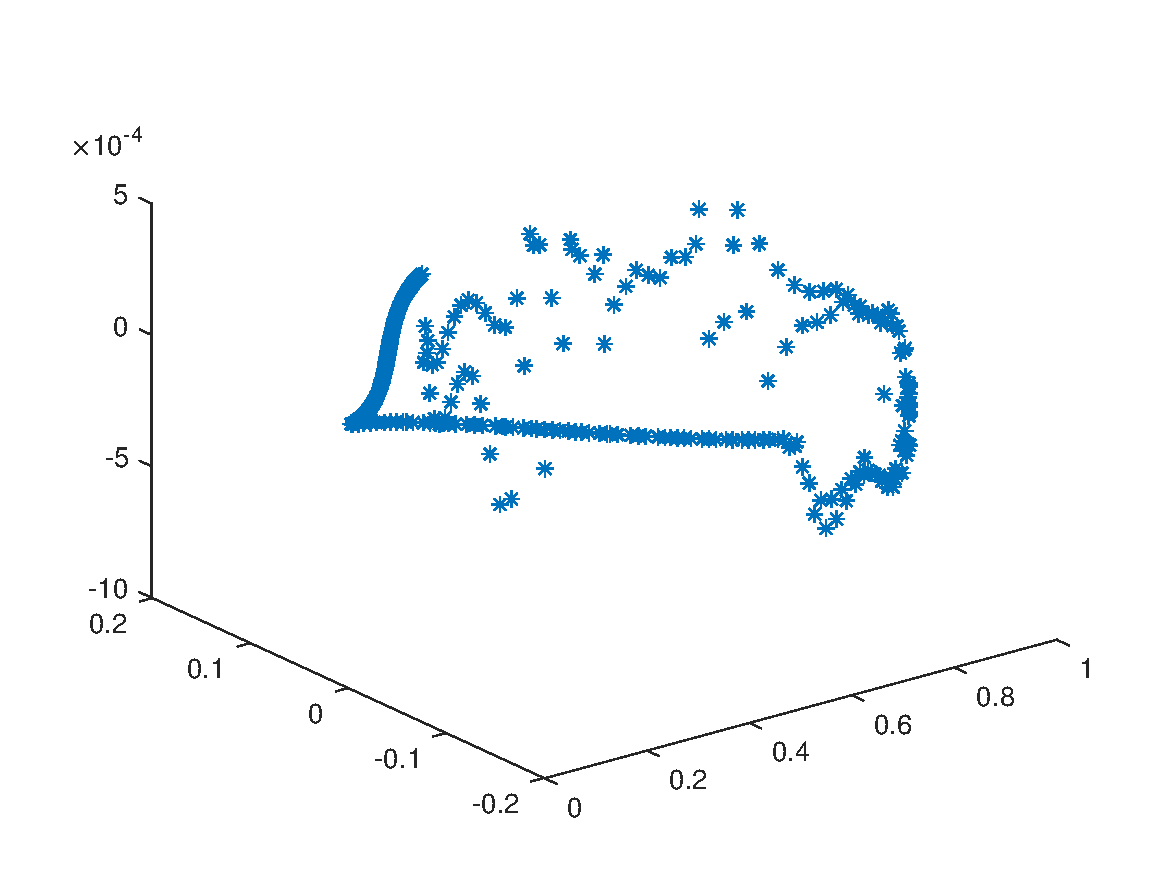}}
	\caption{~The error of $y^{*}, p^{*}, u^{*}$ (without constraint) with N=128.}
	\label{13yserror}
\end{figure}

Also, the numerical solution images of the state, the adjoint state and the control with $N=128$ are displayed in Figures $\ref{13ysstate}$-$\ref{13yscontrol}$; The error images of the state, the adjoint state and the control with $N=128$ are displayed in Figures $\ref{13yserror}$. Comparing Figures $\ref{13control}$ and $\ref{13yscontrol}$ we can observe that the control constraints indeed serve a useful purpose.

\section{Conclusion}\label{sec_con}
In this paper we studied Neumann interface control problem with elliptic PDE constraints and developed and analyzed an efficient computational method for solving these problems. An analysis of the existence, uniqueness and regularity of the optimal solution was carried out by the first order optimality condition. In numerical approximation of the optimality system, we applied a nonconforming linear immersed finite element method to approximate the state and adjoint state, while the control which acts on the interface is approximated by the piecewise constant function. We considered the interface optimal control problem with discontinuous coefficient subject to the lower regularity of the state equation in convex, polygonal domains. The standard linear finite element method can not achieve optimal convergence when the uniform mesh is used. Optimal error estimates are derived for the control, state and adjoint state, which are the same as that of a Neumann boundary control problem without interfaces. Both approaches of postprocessing and variational discretization are applied for the error estimates of the control on the interface and we prove optimal error estimates of the state and adjoint state, and a better approximation for the error estimates of the control on the interface. Numerical examples demonstrate the correctness of theoretical results. The numerical analysis of other methods for this kind of interface control problems will be considered in the future.

\subsection*{Acknowledgements}

The first author is partially supported by the National Natural Science Foundation of China grant No.11971241. 


\bibliographystyle{abbrv}

\end{document}